\documentclass{amsart}

\pdfoutput=1

\usepackage[utf8]{inputenc}
\usepackage[T1]{fontenc}

\usepackage{amssymb}
\usepackage{amsmath}
\usepackage{lmodern}
\usepackage[english]{babel}

\usepackage{hyperref}

\newtheorem{thm}[subsection]{Theorem}
\newtheorem{lem}[subsection]{Lemma}
\newtheorem{prop}[subsection]{Proposition}
\newtheorem{cor}[subsection]{Corollary}
\newtheorem{hyp}[subsection]{Hypothesis}

\theoremstyle{definition}
\newtheorem{definition}[subsection]{Definition}

\theoremstyle{remark}
\newtheorem{rem}[subsection]{Remark}

\numberwithin{equation}{subsection}

\usepackage[
style=ieee-alphabetic,
giveninits=true,
isbn=false,
url=false,
]{biblatex}
\AtEveryBibitem{%
  \clearfield{note}%
}

\usepackage{dsfont}

\usepackage{xparse}
\usepackage{tikz-cd}
\usetikzlibrary{arrows}
\usepackage{comment}
\usepackage{mathtools}

\def\<{\langle}
\def\>{\rangle}
\def\leq{\leqslant}
\def\geq{\geqslant}

\def\1{\mathds{1}}
\def\a{\alpha}

\def\Ac{\mathcal A}

\def\C{\mathbb C}

\def\e{\epsilon}

\def\G{\mathbb G}
\def\Ga{\Gamma}

\def\K{\mathcal K}
\def\l{\lambda}

\def\m{\mathfrak m}

\def\norm{\mathcal N}
\def\O{\mathcal O}

\def\Q{\mathbb Q}

\def\R{\mathbb R}

\def\s{\sigma}

\def\tra{\mathsf T}

\def\Z{\mathbb Z}
\def\mydash{\protect\nobreakdash-\hspace{0pt}}

\newcommand{\diff}{\mathop{}\!\mathrm{d}}

\newcommand{\iso}{\xrightarrow{\protect{\raisebox{-2pt}[0pt][0pt]{$\sim$}}}}

\makeatletter
\newcommand{\extp}{\@ifnextchar^\@extp{\@extp^{\,}}}
\def\@extp^#1{\mathop{\bigwedge\nolimits^{\!#1}}}
\makeatother

\DeclareMathOperator{\ab}{ab}

\DeclareMathOperator{\Aut}{Aut}

\DeclareMathOperator{\cha}{char}

\DeclareMathOperator{\End}{End}
\DeclareMathOperator{\et}{\text{\'et}}

\DeclareMathOperator{\Frob}{Frob}
\DeclareMathOperator{\Gal}{Gal}
\DeclareMathOperator{\GL}{GL}
\DeclareMathOperator{\Hom}{Hom}

\DeclareMathOperator{\id}{id}
\DeclareMathOperator{\Img}{Im}
\DeclareMathOperator{\Ind}{Ind}

\DeclareMathOperator{\ord}{ord}

\DeclareMathOperator{\Rep}{Rep}

\DeclareMathOperator{\rk}{rk}

\DeclareMathOperator{\spe}{sp}

\DeclareMathOperator{\tame}{t}

\DeclareMathOperator{\Tr}{Tr}

\DeclareMathOperator{\unr}{ur}
\DeclareMathOperator{\Vect}{Vect}

\DeclareMathOperator{\WD}{WD}
\DeclareMathOperator{\W}{W}
\DeclareMathOperator{\wild}{w}

\DeclareDocumentCommand\bigslant{ m m g }{
{\IfNoValueT{#3}{\left.}\raisebox{.25em}{$#1$}\!\IfNoValueT{#3}{\middle/}\IfNoValueF{#3}{#3}\!\raisebox{-.25em}{$#2$}\IfNoValueT{#3}{\right.}}
} 

\newenvironment{smallarray}[1]
 {\null\,\vcenter\bgroup\scriptsize
  \arraycolsep=.23em
  \hbox\bgroup$\array{@{}#1@{}}}
 {\endarray$\egroup\egroup\,\null}

\usepackage{enumitem}
\setlist[enumerate]{label=\upshape(\arabic*)}

\begin{document}

\title[Root numbers and real multiplication]{On the root numbers of abelian varieties with real multiplication}

\author{Lukas Melninkas}
\address{IRMA, Université de Strasbourg \\ 7, rue René Descartes \\ 67084 Strasbourg Cedex, France}
\curraddr{}
\email{melninkas@math.unistra.fr}
\thanks{}

\begin{abstract}
Let $A/K$ be an abelian variety with real multiplication defined over a $p$-adic field $K$ with $p>2$. We show that $A/K$ must have either potentially good or potentially totally toric reduction. In the former case we give formulas of the local root number of $A/K$ under the condition that inertia acts via an abelian quotient on the associated Tate module; in the latter we produce formulas without additional hypotheses. 
\end{abstract}


\date{\today}

\maketitle


\section*{Introduction}

Let $A$ be an abelian variety defined over a number field $\K$. The Hasse--Weil conjecture (see, e.g., \autocite[\S4.1]{serre_Lfun_conj}) predicts that its completed $L$-function $\Lambda(A/\K,s)$ has a meromorphic continuation to the whole of $\C$ and satisfies a functional equation \[\Lambda(A/\K,s)=w(A/\K)\Lambda(A/\K,2-s).\] The coefficient $w(A/\K)$ is called the global root number. A straightforward consequence of the conjecture is the equality $w(A/\K)=(-1)^{\ord_{s=1}\Lambda(A/\K,s)}$. On the other hand, the Mordel--Weil theorem tells us that the group of rational points $A(\K)$ is finitely generated. Its rank $\rk(A/\K)$ is notably hard to compute in general and, granting analytic continuation of $\Lambda$ at $s=1$, is predicted to be equal to $\ord_{s=1}\Lambda(A/\K,s)$ by the Birch and Swinnerton-Dyer conjecture. As a consequence of the two aforementioned conjectures, we get a third one, the Parity Conjecture, which is the equality \[(-1)^{\rk(A/\K)}=w(A/\K).\]

Deligne \autocite{deligne_eq_fonctionelle} shows that the global root number can be defined unconditionally as the product of local factors $\prod_v w(A_v/\K_v)$ where $v$ runs through all the places of $\K$ and all but finitely many factors are $1$. For each $v$, the local root number $w(A_v/\K_v)$ is defined via the respective local complex Weil--Deligne representation by using the theory of $\e$-factors, see \autocite[\S4]{deligne_eq_fonctionelle}. It is expected that the geometric properties of $A$ impose enough conditions on the associated Weil--Deligne representations to allow a complete and explicit determination of the root number. Let us briefly review the existing formulas in the next paragraphs.

For each infinite place, the Weil--Deligne representation is defined using the Hodge decomposition of $H^1(A(\C),\C)$ and the root number is always $(-1)^{\dim A}$ (see \autocite[Lemma~2.1]{sabitova_root}). 

For a finite place $v$ of residual characteristic $p$ the Weil--Deligne representation is obtained via an $\ell$-adic Galois representation for any $\ell\neq p$ (or even for $\ell=p$ following Fontaine). We know that $w(A_v/\K_v)=1$ if $A$ has good reduction at $v$. In general, the existence of the Weil pairing forces $w(A_v/\K_v)=\pm1$. 

Let $K/\Q_p$ be a finite extension and let $A/K$ be an elliptic curve. When $p\geq 5$, the local root number was computed by Rohrlich \autocite{rohrlich_formulas}. In this case, the root number is essentially determined by the Néron--Kodaira reduction type of $A/K$. Kobayashi \autocite{kobayashi} has extended Rohrlich's results to include the case $p=3$, where we find a dependancy on the Artin conductor $a(A/K)$. Kobayashi's general formula also includes some coefficients of a Weierstrass equation. The case $p=2$ has been studied by Connell \autocite{connell} and the Dokchitsers \autocite{dd_root_ellc2}.

For a general abelian variety, a framework for studying the associated Weil--Deligne representation and its root number was developed by Sabitova \autocite{sabitova_root}. Under the condition that the inertia action on the Tate module is tame, Bisatt \autocite{bisatt} gives explicit formulas. 

\subsection{The setup and results} Let $p$ be a prime and let $K/\Q_p$ be a finite extension. Let $A/K$ be abelian variety of dimension $g$ together with a polarization $\l\colon A\to A^\vee$ defined over $K$. Let $\cdot^\dag$ denote the corresponding Rosati involution on $\End_K^0(A):=\End_K(A)\otimes_{\Z}\Q$. If there exists a number field $F$ of degree $[F:\Q]=g$ and an inclusion of rings $F\to\End^0_K(A)$ such that the image of $F$ is fixed by $\cdot^\dag$, then we say that $A$ has real multiplication (RM) by $F$ over $K$. In this case, the positivity of $\cdot^\dag$ implies that $F$ is totally real (see, e.g., \autocite[p.~41, Lemma~2]{shimura_taniyama}). We will discuss some examples in \ref{subs:examples}. The main results of this paper generalize the methods used in the case of elliptic curves for the RM setting and produce formulas for local root numbers $w(A/K)$ that extend previously known results. 

Let $\Ga_K=\Gal(\overline{K}/K)$, $I_K$, and $q_K$ be respectively the absolute Galois group, the inertia subgroup, and the order of the residue field of $K$. We fix a prime $\ell\neq p$ and we denote by $\rho_\ell$ the $\ell$-adic Galois representation of $\Ga_K$ on $H^1_{\et}(A_{\overline{K}},\Q_\ell)$. By $a(A/K)$ we denote the Artin conductor of $\rho_\ell$, which is independent of $\ell$. We shall prove the following.

\begin{thm}[{Cor.~\ref{cor:exposition_pot_good}}]\label{thm:exposition_pot_good}
Let $A/K$ be an abelian variety of dimension $g$ with real multiplication. Let us suppose that $p\neq2$ and that $A/K$ has potentially good reduction, so that $\rho_\ell(I_K)$ is a finite group of some order $ep^r$ with $p\nmid e$.
\begin{enumerate}
\item If $\rho_\ell(\Ga_K)$ is abelian, then $\frac{q_K-1}{e}$ is an integer, and \[w(A/K)=(-1)^{\frac{g(q_K-1)}{e}};\]
\item If $\rho_\ell(\Ga_K)$ is non-abelian and $\rho_\ell(I_K)$ is abelian, then $\frac{a(A/K)}{2g}$ and $ \frac{q_K+1}{e}$ are integers, and we have \[w(A/K)=(-1)^{\frac{a(A/K)}{2}+\frac{g(q_K+1)}{e}}.\]
\end{enumerate}
\end{thm}

\begin{thm}[{Prop.~\ref{prop:geom_trichotomy} and Cor.~\ref{cor:real_potMult_rootN_formulas}}] \label{thm:exposition_pot_toric}
Let $A/K$ be an abelian variety with real multiplcation. We suppose that $p\neq2$ and that $A/K$ does not have potentially good reduction. Then $A/K$ has potentially totally toric reduction and  
\[w(A/K)=\begin{cases}(-1)^g & \text{if the reduction (over $K$) is split multiplicative};\\ 
1 & \text{if the reduction is non-split multiplicative};\\ (-1)^{g\frac{q_K-1}{2}}&\text{if the reduction is additive}.\end{cases}\]
In addition, the three possibilities in the above formula are the only ones. 
\end{thm}

\begin{thm}[{Prop.~\ref{prop:odd_dim_rootn}, Cor.~\ref{cor:real_potMult_rootN_formulas}}] \label{thm:exposition_local_even}
If $A/K$ is an abelian variety of even dimension with real multiplication , then $w(A/K)=1$. 
\end{thm}

\begin{cor}\label{cor:exposition_global_even}
Let $\K$ be a number field. If $A/\K$ is an abelian variety with real multiplication and $\dim A$ is even, then $w(A/\K)=1$. 
\end{cor}

\begin{proof}
The local root number at each finite place of $\K$ is $1$ by Thm~\ref{thm:exposition_local_even}. At infinite places they are also $1$ by \autocite[Lemma~2.1]{sabitova_root}.
\end{proof}

\begin{rem}\label{rem:real_m_glob}
Let $\K$ be a number field and let $A/\K$ be an abelian variety with real multiplication by $F$. For each finite place $v$ of $\K$, we have a decomposition of associated complex WD-representations $\WD_i(\rho_\ell(A_v/\K_v))=\prod_{\iota}\rho_{v,\iota}'$, where the product is taken over all embeddings $\iota\colon F\hookrightarrow \C$. Each family $(\rho_{v,\iota}')_\iota$ is composed of $\Aut(\C)$-conjugate representations (we prove this for the potentially good reduction case in Prop.~\ref{prop:rho_iota}.(1) and for the remaining cases in \ref{subs:pot_mult_F-rationality}). The root numbers $w(\rho'_{v,\iota},\psi_{\K_v})$ are independent of $\iota$ by \autocite[Thm.~1]{rohrlich_galois_inv}. We choose some $\iota$, fix any place $v$ of $\K$, and define \[w(\iota A_v/\K_v):=\begin{cases} w(\rho'_{v,\iota},\psi_{\K_v}) &\text{if }v\text{ is finite,} \\ -1 & \text{if }v\text{ is infinite.}\end{cases}\] Furthermore, let us define \[w(\iota A/\K):=\prod_v w(\iota A_v/\K_v).\] We have $w(A/\K)=w(\iota A/\K)^{\dim A}$. The number $w(\iota A/\K)$ appears, for example, in the functional equation of a certain $L$-function (see \autocite[(0.1) and \S4.9]{nekovar_comp_l}) and, consequently, in a certain version of the $p$-Parity Conjecture (now a theorem by Neková\v{r}~\autocite[Thm.~D]{nekovar_comp_bt}).
\end{rem}

\subsection{Examples}\label{subs:examples}
 \begin{enumerate}
\item Every elliptic curve has RM by $\Q$. For every case of Theorems~\ref{thm:exposition_pot_good} and \ref{thm:exposition_pot_toric} there is an elliptic curve that falls under that setting. Given an elliptic curve $E/K$, let $\Delta$ be any of its discriminants. The curve $E/K$ has potentially good reduction if and only if its $j$-invariant is in $\O_K$. In this case, the invariant $e$ can be determined from the residue $v_K(\Delta)\bmod 4$, and $\rho_\ell(\Ga_K)$ is abelian if and only if $\Delta\in (K^\times)^2$, see \autocite{kraus}. Furthermore, by applying Tate's algorithm we can compute the Artin conductor $a(E/K)$ from a Weierstrass equation. 
\item The database \autocite{lmfdb} can be used to find many examples of hyperelliptic curves over $\Q$ of genus $2$ whose Jacobians $J/\Q$ have RM. Although the local root numbers are all $w(J_p/\Q_p)=1$, the signs $w(\rho'_{p,\iota},\psi_{\Q_p})$ as in Remark~\ref{rem:real_m_glob} can also be computed by using the theorems of our paper in the case where $\rho_\ell(I_{\Q_p})$ is abelian.
\item Some explicit families of hyperelliptic curves whose Jacobians have RM have been given in \autocite{mestre_real_m}, \autocite{tautz_real_m}. A more general method to construct curves with RM has been presented in \autocite{ellenberg_end}. However, in order to apply our results for these examples, one needs to study the associated $\ell$-adic Galois representations more closely, which is beyond the scope of our paper. 
\item More generally, the classical theory due to Shimura (see, e.g., \autocite[\S9.2]{ab_var_cmplx}) produces moduli spaces of abelian varieties with RM by a prescribed totally real number field. 
\end{enumerate}

\subsection{The structure of the paper} In Section~\ref{sect:l_adic_repr} we recall the general theory of local Weil--Deligne representations and root numbers associated to abelian varieties. In Section~\ref{sect:ab_var_action} we present some properties of $\rho_\ell(A/K)$ implied by the structure of $\End_K(A/K)$ for a general field $K$; we prove some geometric restrictions in the RM case over a $p$-adic field. These results are probably known to specialists but are somewhat difficult to find in the existing literature, so we include the details for the sake of completeness. In Section~\ref{sect:real_AV_potG_red} we continue the study by assuming potentially good reduction. The results of Sections~\ref{sect:ab_var_action} and~\ref{sect:real_AV_potG_red} are then used in Section~\ref{sect:ab_inertia} to prove Thm.~\ref{thm:exposition_pot_good} and in Section~\ref{sect:real_AV_pot_mult} to prove Thm.~\ref{thm:exposition_pot_toric}. 

\subsection{Notation and conventions}

Given a field $K$, we fix a separable closure $\overline{K}$ and denote the absolute Galois group by $\Ga_{K}:=\Gal(\overline{K}/K)$, equipped with the Krull topology. We will suppose implicitly that every separable extension of $K$ mentioned in the text lies inside $\overline{K}$. Note that if $K$ is perfect, then $\overline{K}$ is an algebraic closure of $K$. For a topological ring $R$, we denote by $\Rep_R(\Ga_K)$ the category of continuous $R$-linear finite-rank representations of $\Ga_K$.

Let $K$ be a field equipped with a valuation $v_K$. We write $\O_K$, $\m_K$, and $k_K$ for, respectively, the ring of integers associated to $v_K$, the maximal ideal of $\O_K$, and the residue field $\O_K/\m_K$. Whenever $v_K$ is discrete, we suppose that $v_K$ is normalized, i.e., that $v_K(K^\times)=\Z$. An element $\varpi_K\in K$ of valuation one is called a uniformizer of $K$. 

If $K$ is a local field, then, in particular, $k_K$ is a finite field of some order $q_K$. In this case we have a canonical topological generator $\Frob_{k_K}\colon x\mapsto x^{q_K}$ of $\Ga_{k_K}$, called the arithmetic Frobenius element. Let $\pi\colon\Ga_K\twoheadrightarrow \Ga_{k_K}$ be the surjection induced by an isomorphism $\O_{\overline{K}}/\m_{\overline{K}}\simeq \overline{k}_K$, and denote by $\varphi_K$ a geometric Frobenius lift, i.e., an element in $\Ga_K$ which is sent to $\Frob_{k_K}^{-1}$ via $\pi$. 

For an abelian variety $A/K$ defined over a field $K$, let $A^\vee/K$ be its dual variety, and let $\End_K^0(A):=\End_K(A)\otimes_\Z \Q$. When $K$ is the fraction field of a Dedekind domain $\mathcal{D}$, we denote by $\Ac$ the corresponding Néron model over $\mathcal{D}$, by $\Ac_k$ the special fiber over a residue field $k$, and by $\Ac_k^0$ the identity component of that special fiber. 

\section{From Galois representations to root numbers}\label{sect:l_adic_repr}

Let $p$ be any prime number and let $K/\Q_p$ be a finite extension. We give a brief summary of the general theory of local root numbers, which we will define via Weil--Deligne representations. 

\subsection{Structure of the absolute Galois group}\label{subs:struct_galois} 

The elements of $\Ga_K$ that induce the trivial automorphism of $k_{\overline{K}}$ define a normal closed subgroup $I_K$, the inertia subgroup of $\Ga_K$, which cuts out the maximal unramified extension $K^{\unr}:=\overline{K}^{I_K}$ of $K$. The quotient $\Ga_K/I_K\cong \Ga_{k_K}$ is the profinite completion $\widehat{\<\Frob_{k_K}\>}$ of the infinite cyclic group generated by the Frobenius element. Let $\varphi_K\in \Ga_K$ be a lift of the geometric Frobenius. The closure of the subgroup generated by $\varphi_K$ is the subgroup $\widehat{\<\varphi_K\>}\subset \Ga_K$ isomorphic to $\Ga_{k_K}$. Thus, we have a splitting $\Ga_K=I_K\rtimes\widehat{\<\varphi_K\>}$. 

On the other hand, the elements acting trivially on $\overline{K}^\times\!\!\big/(1+\m_{\overline{K}})$ define a normal pro-$p$ subgroup $I^{\wild}_K$ of $\Ga_K$, called the wild inertia subgroup, which cuts out the maximal tamely ramified extension $K^{\tame}/K$. The tame inertia group is the quotient $I_K^{\tame}:= I_K/I_K^{\wild}$, which is isomorphic to the pro-$p$-complementary completion $\prod_{p'\neq p}\Z_{p'}$ of $\Z$. For every $j\in I_K$ we have $\varphi_K^{-1} j \varphi_K\equiv j^{q_K}\bmod I^{\wild}_K$. An application of the profinite version of the Schur--Zassenhaus theorem (see, e.g., \autocite[Thm.~2.3.15]{rz}) shows that $I^{\tame}_K$ lifts to a profinite subgroup of $I_K$, thus giving a semidirect product structure on $I_K$. Each lift of the tame inertia is determined by a choice of a topological generator $\tau_K\in I_K$. As a consequence of the above discussion, we have \begin{equation}\label{eq:galois_split_easy}\Ga_K=\left(I^{\wild}_K\rtimes \widehat{\<\tau_K\>}'\right)\rtimes \widehat{\<\varphi_K\>},\end{equation}
where $\widehat{\phantom{t}\cdot\phantom{t}}'$ denotes the pro-$p$-complementary completion. We will need the following more refined result. 

\begin{thm}[Iwasawa]\label{thm:galois_split_hard}
For every lift $\tau_K\in I_K$ of a topological generator of $I^{\tame}_K$ we can choose $\varphi_K$ so that we have \[\Ga_K=I_K^{\wild}\rtimes\left(\widehat{\<\tau_K\>}'\rtimes \widehat{\<\varphi_K\>}\right),\] and $\varphi_K^{-1}\tau_K\varphi_K=\tau_K^{q_K}$.
\end{thm}

\begin{proof}
This is essentially \autocite[Thm.~2]{iwasawa}. We first choose any lifts $\varphi_K\in\Ga_K$ and $\tau_K\in I_K$, for which \eqref{eq:galois_split_easy} holds. Following \autocite[Lemma~4]{iwasawa}, we can modify $\varphi_K$ so that $\varphi_K^{-1}\cdot\widehat{\<\tau_K\>}'\cdot\varphi_K=\widehat{\<\tau_K\>}'$. Then, $\varphi_K^{-1}\tau_K\varphi_K$ is in $ (\tau_K^{q_K}I_K^{\wild})\cap \widehat{\<\tau_K\>}'=\{\tau_K^{q_K}\}$. The closed subgroup $H:=\widehat{\<\tau_K\>}'\cdot \widehat{\<\varphi_K\>}\subset\Ga_K$ has the desired semi-direct product structure. Also, we have $I_K^{\wild}H=\Ga_K$, and $H\cap I_K^{\wild}$ is trivial, thus proving the theorem.
\end{proof}

\subsection{Weil groups and representations} Let $K'/K$ be a Galois extension such that $K^{\unr}\subseteq K'$. We have an isomorphism $k_{K'}\simeq \overline{k}_K$. The arguments used to obtain \eqref{eq:galois_split_easy} apply similarly to show that the exact sequence \[0\to I(K'/K)\to \Gal(K'/K)\xrightarrow{\pi} \Ga_{k_K}\to 0\] splits, giving a semidirect product structure 
\begin{equation}\label{eq:galois_split}\Gal(K'/K)=I(K'/K)\rtimes\widehat{\<\varphi_K\>}\end{equation}
where $I(K'/K):=I_{K}/I_{K'}\subset \Gal(K'/K)$ is the inertia subgroup. The Weil group of $K'/K$ is defined as \[W(K'/K):=\pi^{-1}(\<\Frob_{k_K}\>).\] In other words, $W(K'/K)$ consists of the elements of $\Gal(K'/K)$ that $\pi$ sends to $\Frob_{k_K}^{n}$ for some $n\in\Z$. We note that 
\begin{equation}\label{eq:weil_gr_split} W(K'/K)=I(K'/K)\rtimes\<\varphi_K\>.\end{equation}
 We equip $W(K'/K)$ with the topology generated by the open subgroups of $I(K'/K)$ and their translates. We denote $W_K:=W(\overline{K}/K)$. 

Let $F$ be a field of characteristic zero. An $F$-linear Weil representation is defined as a continuous representation $\rho\colon W_K\to\GL(V)$ where $V$ is a finite-dimensional $F$-vector space, and $\GL(V)$ is equipped with the discrete topology. Continuity is equivalent to $\rho(I_K)$ being a finite group. Let us note that if $F=\C$, then we get an equivalent definition if we equip $\GL(V)$ with the standard topology (see \autocite[\S2]{rohrlich}). The category of $F$-linear Weil representations over $K$ will be denoted by $\Rep_F(W_K)$. Every $f\in\End(V)$ gives the transpose endomorphism $f^\tra\in\End(V^*)$ on the dual vector space. The dual representation of $\rho\in\Rep_F(W_K)$ is given by $\rho^*(g):=\rho(g^{-1})^\tra$ for every $g\in W_K$. 

\subsection{Unramified Weil characters}\label{subs:unram_char} Let $R$ be a ring in which $p$ is invertible. We define the cyclotomic character $\omega_K\colon W_K\to R^{\times}$ to be the unramified character such that $\omega_K(\varphi_K)=q_K^{-1}$. When $R=\C$, for any $z\in\C$, we define the unramified complex Weil character $\omega_K^z$ by taking $b=-\ln(q_K)$ and then setting $\omega_K^z(\varphi_K)=\exp(zb)$. We note that any unramified complex Weil character is of the form $\omega_K^z$ for some $z\in\C$. For a complex Weil representation $\rho$ on $V$, we define the Tate twist $\rho(z):=\rho\otimes\omega_K^z$ for any $z\in \C$ and denote the underlying vector space by $V(z)$. 

\subsection{Weil--Deligne representations} Let $F$ be a field of characteristic ze\-ro. An $F$-linear Weil--Deligne (WD for short) representation is a pair $(\rho,N)$ where $\rho$ is an $F$-linear Weil representation on $V$ and $N$ is a nilpotent $F$-endomorphism of $V$, called the monodromy operator, satisfying $\rho N \rho^{-1}=\omega_K N$. The $F$-linear WD-representations form a category $\Rep_F(W'_K)$ where the morphisms between two WD-representations are the morphisms between their respective Weil representations that commute with the monodromy operators. Trivially, there is an equivalence between the category $\Rep_F(W_K)$ and the full subcategory of $F$-linear WD-representations with trivial monodromy. If $\rho'=(\rho,N)$ and $\s'=(\s,P)$ are two WD-representations, then we define their direct sum as $\rho'\oplus\s'=(\rho\oplus\s, N\oplus P)$, their tensor product as $\rho'\otimes\s'=(\rho\otimes\s, N\otimes \id+\id\otimes P)$, and the dual representation of $\rho'$ as $(\rho')^*=(\rho^*, -N^{\tra})$.

\begin{rem}
As in \autocite[\nopp 8.3.6]{deligne_eq_fonctionelle}, one may define the Weil--Deligne group $W_K'$ as a group scheme over $\Q$ that is the semidirect product of $W_K$ by $\G_a$ where for every $\Q$-algebra $R$, every $x\in \G_a(R)$, and every $\s\in W_K$ we have $\s x \s^{-1}=\omega_K(\s)x$. A representation of the group scheme $W_K'$ over a field of characteristic 0 may be shown to correspond to a pair $(\rho,N)$ as above. 
\end{rem}

\subsection{$\ell$-adic monodromy}\label{subs:l_adic_monodromy} Let $\ell\neq p$ be a prime number, $\varphi_K$ a geometric Frobenius lift, and $t_\ell\colon I_K\to \Q_\ell$ a nontrivial continuous homomorphism. Grothendieck's $\ell$-adic monodromy theorem (see \autocite[\nopp 8.4.2]{deligne_eq_fonctionelle}) provides a fully faithful functor
\begin{align*}\WD\colon \Rep_{\Q_\ell}(\Ga_K) & \to \Rep_{\Q_\ell}(W'_K) \\ \rho_\ell &\mapsto (\W(\rho_\ell), N_{\rho_\ell}).\end{align*} The construction of $\W(\rho_\ell)$ depends on the choice of $\varphi_K$, and $N_{\rho_\ell}$ is the unique nilpotent endomorphism such that $\rho_\ell(j)=\exp(t_\ell(j) N_{\rho_\ell})$ for every $j$ in a sufficiently small open subgroup of $I_K$. In particular, if $\rho_\ell(I_K)$ is finite, then $N_{\rho_\ell}=0.$ By \autocite[\nopp 8.4.3]{deligne_eq_fonctionelle}, the isomorphism class of $\WD(\rho_\ell)$ is independent of the choices of $t_\ell$ and $\varphi_K$. 

\subsection{Complex WD-representations}\label{subs:complex_WD} Let $\rho_\ell\in \Rep_{\Q_\ell}(\Ga_K)$. We fix an embedding $i\colon \Q_\ell\to \C$ for the rest of the text. By extending the scalars of $\WD(\rho_\ell)$ to $\C$ via $i$ we obtain a complex WD-representation $\WD_i(\rho_\ell)=(\W_i(\rho_\ell), N_{i,\rho_\ell})$. As the notation indicates, the isomorphism class of $\WD_i(\rho_\ell)$ generally depends on $i$.

\subsection{Geometric $\ell$-adic representations}\label{subs:l_adic_rep_av} Let $K$ be any field with a separable closure $\overline{K}$, and let $\ell\neq\cha(K)$ be a prime. If $X/K$ is a smooth and proper variety, then we have a natural $\ell$-adic Galois representation on the ($m$-th) $\ell$-adic cohomology group $H_{\et}^m(X_{\overline{K}},\Z_\ell)\otimes_{\Z_\ell}\Q_\ell$ for each $m\geq 0$. For a commutative algebraic group $G$ over $K$, one defines its $\ell$-adic Tate module as the projective limit of $\overline{K}$-valued $\ell^n$-torsion points \[T_\ell G=\varprojlim_n G(\overline{K})[\ell^n]\] equipped with the $\ell$-adic topology. The natural $\Ga_K$-action on $T_\ell G$ is continuous and defines a $\Q_\ell$-linear representation on $V_\ell G:=T_\ell G\otimes_{\Z_\ell}\Q_\ell$. For an abelian variety $A/K$ of dimension $g$, the $\Z_\ell$-module $T_\ell A$ is free of rank $2g$, and its dual $(T_\ell A)^*$ is canonically isomorphic to $H^1_{\et}(A_{\overline{K}},\Z_\ell)$ in $\Rep_{\Z_\ell}(\Ga_K)$. We will denote the corresponding $\ell$-adic Galois representation on $(V_\ell A)^*$ by $\rho_\ell(A/K)$ or simply by $\rho_\ell$ when the context is clear. 

The multiplicative group $\G_m/K$ induces a $\Z_\ell$-linear Galois representation on $T_\ell \G_m$ of dimension one, which corresponds to a character $\omega_K\colon\Ga_K\to \Z_\ell^\times$. For every $n\in\Z$ and every $\ell$-adic representation $\rho$ on $V$, we define the Tate twist by $\rho(n):=\rho\otimes\omega^n_K$ and denote its underlying space by $V(n)$. 

For the rest of Section~\ref{sect:l_adic_repr} we return to the case where $K$ is a finite extension of $\Q_p$ and $\ell\neq p$. Restricting the $\Ga_K$-action on $T_\ell\G_m$ to $W_K$ we obtain exactly the unramified character $\omega_K\colon W_K\to\Z_\ell^\times$ of \ref{subs:unram_char}. 

\subsection{Reduction of abelian varieties} Let $K/\Q_p$ be a finite extension. An abelian variety $A/K$ is said to have good (resp. semistable, synonymous with semi-abelian) reduction if the identity component $\Ac_{k_K}^0$ of the special fiber of the Néron model is an abelian variety (resp. a semi-abelian variety). More generally, the reduction is potentially good (resp. potentially semistable) if there exists a finite extension $L/K$ such that $A/L$ has good (resp. semistable) reduction. We will say that $A/K$ has bad reduction if it does not have good reduction. Recall the following well-known results. 

\begin{thm}\label{thm:ab_var_reduction} An abelian variety $A/K$:
\begin{enumerate}
\item has good reduction if and only if $\rho_\ell$ is unramified, i.e., if and only if $I_K$ acts trivially on $V_\ell A$ (Néron--Ogg--Shafarevich criterion);
\item has semistable reduction if and only if $\rho_\ell|_{I_K}$ is unipotent (Gro\-then\-dieck's inertial criterion);
\item always has potentially semistable reduction (Grothendieck's semistable reduction theorem).
\end{enumerate}
\end{thm}

\begin{proof}
For (1) see \autocite[Thm.~1]{serre_tate}; for (2) see \autocite[p.~350, Prop.~3.5]{sga7_i}; for (3) see \autocite[p.~21, Thm.~6.1]{sga7_i}.
\end{proof}

\begin{definition}\label{def:im}
Given an abelian variety $A/K$ with potentially good reduction, an extension $L'/K$ will be called \emph{inertially minimal (IM) for} $A/K$ if $L'K^{\unr}=M$ where $M/K^{\unr}$ is the extension cut out by $\ker\rho_\ell|_{I_K}$.  
\end{definition}

\begin{lem}\label{lem:min_ext_good} Let $A/K$ be an abelian variety with potentially good reduction. Let $M/K$ be as in Def.~\ref{def:im}. Then 
\begin{enumerate} 
\item An extension $L'/K$ is IM for $A/K$ if and only if $I_{L'}=\ker\rho_\ell|_{I_K}=I_M$. This is also equivalent to the condition that $A$ has good reduction over $L'$ and has bad reduction over any smaller extension $L''/K$ such that $L'/L''$ is ramified. 
\item The extension $M/K$ is Galois; 
\item For any choice of a Frobenius lift $\varphi_K$, there exists a finite totally ramified extension $L'/K$ which is fixed by $\varphi_K$ and is IM for $A/K$; 
\item Let $L'/K$ be an IM-extension for $A/K$ and let $L/K$ be its Galois closure in $\overline{K}$. Then $L/K$ is also IM for $A/K$. In particular, $L/L'$ is unramified. 
\end{enumerate}
\end{lem}

\begin{proof}
The first equivalence of (1) follows from Galois theory, and the second is a reformulation of the Néron--Ogg--Sha\-fa\-re\-vich criterion. For (2) we observe that $\ker\rho_\ell|_{I_K}=I_K\cap \ker\rho_\ell$, which is normal in $\Ga_K$. For any $\varphi_K$, the subgroup $I_M\cdot{\widehat{\vphantom{i}\smash{\<\varphi_K\>}}}\subseteq \Ga_K$ (cf.~\eqref{eq:galois_split_easy}) is closed (as a product of two compact subgroups) and has finite index, thus is open. It cuts out a finite totally ramified extension $L'/K$ with $I_{L'}=I_M$, so (3) follows. The part (4) follows from the observation that $L\subset M$ (since $M/K$ is Galois), which implies that $I_{L'}=I_M\subseteq I_L\subseteq I_{L'}$.\end{proof}

\subsection{$p$-adic uniformization}\label{subs:p_adic_unif} Let $A/K$ be an abelian variety of dimension $g$. As stated in \autocite[Prop.~3.1]{chai_semiab}, there exists a semi-abelian variety $E/K$ of dimension $g$, defined by a Raynaud extension $0\to T\to E\to B\to 0$ where
\begin{enumerate}[label=(\roman*)]
\item $T$ is a torus over $K$,
\item $B$ is an abelian variety over $K$ with potentially good reduction,
\item the rigid analytification of $A/K$ is the rigid analytic quotient of the analytification of $E/K$ by a free Galois sub-$\Z$-module $M\subset E(\overline{K})$ of rank $r=\dim T$. 
\end{enumerate}

We denote $(\kappa,0)=\WD_i(\rho_\ell(B/K))$, note that the monodromy is trivial by Thm.~\ref{thm:ab_var_reduction}.(1) and \ref{subs:l_adic_monodromy}. Let $X(T):=\Hom_{\overline{K}}(T,\G_m)$, and let $\eta\colon W_K\to \GL(X(T)_\Q)$ be the representation given by the Galois action on $X(T)_\Q:=X(T)\otimes_{\Z}\Q$, which has finite image. 

\begin{prop}[{\autocite[Prop.~1.10]{sabitova_root}}]\label{prop:sabitova}
There is an isomorphism of complex  WD\mydash{}representations \[\WD_i(\rho_\ell(A/K))\cong(\kappa, 0)\oplus \left(\eta(-1)\otimes\spe(2)\right),\] where $\spe(2)=\left(\1\oplus \omega_K, \left( \begin{smallmatrix} 0&0\\ 1&0 \end{smallmatrix} \right)\right)$ is the so-called special WD-representation. The representations $\W_i(\rho_\ell(B/K))$ and $\W_i(\rho_\ell(A/K))$ are semisimple, and the isomorphism classes of $\WD_i(\rho_\ell(B/K))$ and $\WD_i(\rho_\ell(A/K))$ are independent of $\ell$ and $i$.
\end{prop}


\begin{rem}\label{rem:p-adic_unif_funct}
\begin{enumerate}
\item[] 
\item The Galois modules $X(T)_\Q$ and $M\otimes\Q$ are isomorphic, see \autocite[Lemma~1.11]{sabitova_root}. 
\item Let us recall a few ideas behind the construction of $T$ in order to give a direct description of $\eta$. By Thm.~\ref{thm:ab_var_reduction}.(3), there exists a finite Galois extension $L/K$ such that the identity component $\Ac_{k_L}^\circ$ of the special fiber of the Néron model of $A/L$ is a semi-abelian variety. By replacing $L$ with a larger extension, we may suppose that the maximal subtorus $\widetilde{T}$ of $\Ac_{k_L}^0$ is split. By the universal property of the Néron model, the $k_L$-scheme $\widetilde{T}$ is equipped with a $k_L$-semilinear $\Gal(L/K)$-action. The resulting $\Ga_K$-representation on $X(\widetilde{T})_{\Q}=\Hom_{k_L}(\widetilde{T},\G_m)\otimes_{\Z}\Q$ is isomorphic to $\eta$.  
\end{enumerate}
\end{rem}

\subsection{Local reciprocity} Local class field theory provides a reciprocity isomorphism for any finite Galois extension $L/K$ of local fields: \[\theta_{L/K}:\bigslant{K^\times}{\norm_{L/K}(L^\times)}\iso \Gal(L/K)^{\ab}.\] There exist exactly two canonical choices for $\theta_{L/K}$, and they are inverses of each other. Following Deligne, we normalise $\theta_{L/K}$ by supposing that it sends the class of a uniformizer of $K$ to the class of a geometric Frobenius lift.  
Using the projective limit of these reciprocity maps as well as the Existence Theorem we obtain (see, e.g., \autocite[Cor.~2.5]{knapp}) an isomorphism \[\theta_K:K^\times\iso W_K^{\ab}.\] 

\subsection{The epsilon factor}\label{subs:eps_factor} We fix a nontrivial locally constant additive character $\psi_K\colon K\!\to\!\C^{\times}$ and a Haar measure $\diff x_K$ on $K$. We let $n(\psi_K):=\max\{n\in\Z\,|\, \psi_K(\m_K^{-n})=1\}$. For a complex WD-representation $\rho'=(\rho,N)$ with underlying vector space $V$, we define the $\e$-factor of $\rho'$ as the product \[\e(\rho',\psi_K,\diff x_K):=\e(\rho,\psi_K,\diff x_K)\delta(\rho'),\] where \[\delta(\rho'):=\det\left(-\rho(\varphi_K)\Big|\bigslant{V^{I_K}\!\!}{(\ker N)^{I_K}}\!\!\right)\] and $\e(\cdot,\psi_K,\diff x_K)$ is the unique function satisfying the following axioms (see \autocite[Théo\-rème~4.1]{deligne_eq_fonctionelle}):
\begin{enumerate}[label=(\roman*)]
\item The map $\e(\cdot,\psi_K,\diff x_K)\colon \Rep_\C(W_K)\to\C^\times$ is multiplicative in short exact sequences of Weil representations, i.e., it induces a group homomorphism from the Grothendieck group of virtual complex Weil representations to $\C^\times$. 
\item For any finite extension $L/K$, any $\diff x_L$, and any $\rho\in\Rep_\C(W_L)$ we have \[\epsilon(\Ind_{W_L}^{W_K}\!\rho,\psi_K,\diff x_K)=\epsilon(\rho,\psi_K\circ\Tr_{L/K},\diff x_L)\!\left(\!\frac{\epsilon(\Ind_{W_L}^{W_K}\!\1_{W_L},\psi_K,\diff x_K)}{\epsilon(\1_{W_L},\psi_K\circ\Tr_{L/K},\diff x_L)}\!\!\right)^{\!\dim\rho}\!\!.\]
\item For any finite extension $L/K$, any $\psi_L$, any $\diff x_L$, and any one\mydash{}dimensional $\xi\in\Rep_{\C}(W_L)$ with Artin conductor $a(\xi)$ we have  \[\epsilon(\xi,\psi_L,\diff x_L)=\begin{cases}
\int_{c^{-1}\O_L^\times}\xi^{-1}(\theta_L(x))\psi_L(x)\diff x_L &\text{ if $\xi$ is ramified},\\
(\xi\omega_L^{-1})(\theta_L(c))\int_{\O_L}\diff x_L &\text{ if $\xi$ is unramified},
\end{cases}\] 
where $c\in L$ is any element of valuation $n(\psi_L)+a(\xi)$.
\end{enumerate} 
The condition (iii) expresses the fact that the $\e$-factor is the coefficient in Tate's local functional equation of the $L$-function of $\xi$, as described in \autocite[\S3]{tate_background}.

\subsection{The root number}\label{subs:root_n} Given a $\rho'=(\rho,N)\in\Rep_{\C}(W'_K)$, we define its root number as \[w(\rho',\psi_K):=\frac{\e(\rho',\psi_K,\diff x_K)}{|\e(\rho',\psi_K,\diff x_K)|}.\] It follows from basic properties of $\e$-factors (see \autocite[\S11, Prop.]{rohrlich}) that the root number $w(\rho',\psi_K)$ is independent of the choice of $\diff x_K$, and that it is also independent of $\psi_K$ if $\det\rho$ is real positive, which is always the case for a WD-representation obtained from an abelian variety.

\subsection{Some custom conventions}\label{subs:conventions_root_n} The dependence of $\e$-factors on $\psi_K$ and $\diff x_K$ will not be important in what follows, so we make some particular choices for the rest of the paper. We denote by $\mu_{p^{\infty}}$ the group of all complex roots of unity of orders that are powers of $p$, and we fix an isomorphism $\Q_p/\Z_p \simeq \mu_{p^{\infty}}$. For a finite extension $K/\Q_p$ we define $\psi_K$ as the composition \[\psi_K\colon K\xrightarrow{\smash{\Tr}}\Q_p\to \Q_p/\Z_p\iso \mu_{p^\infty}\subseteq \C^\times.\] We note that $n(\psi_K)$ is the valuation of any generator of the different ideal $\mathcal{D}_{K/\Q_p}\subseteq \O_K$. We normalize $\diff x_K$ by demanding that $\diff x(\O_K)=1$ .

\begin{prop}\label{prop:root_n_props} Let $\rho'=(\rho,N)\in\Rep_\C(W_K')$, then \begin{enumerate}
\item For any unramified character $\chi$, we have \[\e(\rho'\otimes\chi,\psi_K,\diff x_K)=\e(\rho',\psi_K,\diff x_K)\cdot \chi(\varphi_K)^{n(\psi_K)\cdot \dim \rho' +a(\rho')};\] In particular, for any $s\in\R$, we have $w(\rho'(s),\psi_K)=w(\rho',\psi_K)$; 
\item We have $w(\rho',\psi_K)w((\rho')^*,\psi_K)=\det(\rho(\theta_K(-1)));$
\item If $\rho$ has a finite image and is self-dual, then \[w(\rho\otimes\spe(2),\psi_K)=\det(\rho(\theta_K(-1)))\cdot (-1)^{\<\rho|\1\>},\] where $\<\cdot |\cdot\>$ denotes the usual inner product on characters of finite groups. 
\end{enumerate}
\end{prop}

\begin{proof} (1), (2), and (3) can be deduced from \autocite[\S11, Prop.(iii)]{rohrlich}, \autocite[\S12, Lemma.(iii)]{rohrlich}, and \autocite[p.~327, Prop.~6]{rohrlich_formulas}, respectively. 
\end{proof}

\section{Abelian varieties with an action by a number field}\label{sect:ab_var_action}

Let $A/K$ be an abelian variety of dimension $g$ over some field $K$. We suppose that there exists an inclusion of unitary rings $F\hookrightarrow \End^0_K(A)$ where $F$ is a number field of some degree $n$. We fix a prime number $\ell\neq \cha(K)$. Let us recall that $\rho_\ell$ is the $\Q_\ell$-linear representation of $\Ga_K$ given by the Galois action on $(V_\ell A)^*$.

\subsection{The $F$-action on $(V_\ell A)^*$}\label{subs:F_action} Since the $\Q_\ell$-linear Galois representation $\rho_\ell$ commutes with the action of the semisimple $\Q_\ell$-algebra $F_\ell:=F\otimes_{\Q}\Q_\ell$ on $(V_\ell A)^*$, we may interpret $\rho_\ell$ as an $F_\ell$-linear representation. It is well known (see \autocite[p.~39, Prop.~2]{shimura_taniyama}) that $n|2g$, and we denote $h:=2g/n$. We have a decomposition
\begin{equation}\label{eq:F_ell_decomp}F_\ell=\prod_\l F_\l\end{equation}
where $\l$ runs through all finite places of $F$ above $\ell$, and where $F_\l$ is the corresponding completion of $F$. Consequently, we may decompose 
\begin{equation}\label{eq:V_l-adic_decomp}(V_\ell A)^*=\prod_\l V_\l\end{equation}
where $V_\l=(V_\ell A)^*\otimes_{F_\ell}F_\l$ is an $F_\l$-vector space of some dimension $h_\l$. By $F_\ell$-linearity of $\rho_\ell$, the decomposition \eqref{eq:V_l-adic_decomp} is $\Ga_K$-equivariant, so we obtain a family of Galois representations $\rho_\l\colon \Ga_K\to \Aut_{F_\l}(V_\l).$

\begin{thm}[{\autocite[Thm.~2.1.1]{ribet_real}}]\label{thm:F_mod_free}
In the notation above, $h_\l=h$ for all $\l$, so $V_\ell (A)^* $ is a free $F_\ell$-module of rank $h$.
\end{thm}

\subsection{The $F_\ell$-bilinear Weil pairing}\label{subs:F_weil_pairing} Let $P_{\Q_\ell}\colon V_\ell A\times V_\ell (A^\vee)\to\Q_\ell(1)$ be the $\Ga_K$-equivariant perfect $\Q_\ell$-bilinear Weil pairing. Each $f\!\in\!\End_K^0(A)$ acts on $V_\ell(A^\vee)$ via its dual $f^\vee\!\in\!\End_K^0(A^\vee)$ and satisfies $P_{\Q_\ell}(f(x),y)\!=\!P_{\Q_\ell}(x,f^\vee(y))$ for all $(x,y)\!\in\!V_\ell A\!\times\! V_\ell(A^\vee)$. Then, the pairing induces an isomorphism \[P'_{\Q_\ell}\colon \Hom_{\Q_\ell}(V_\ell A, \Q_\ell)\cong V_\ell(A^\vee)(-1)\] in $\Rep_{\Q_\ell}(\Ga_K)$, which is $\End_K^0(A)\otimes_{\Q}\Q_\ell$-linear and, in particular, $F_\ell$-linear.

We claim that the map \[t\colon \Hom_{F_\ell}(V_\ell A,F_\ell)\to \Hom_{\Q_\ell}(V_\ell A,\Q_\ell)\] given by $f\mapsto \Tr_{F_\ell/\Q_\ell}\circ f$ is an isomorphism in $\Rep_{F_\ell}(\Ga_K)$. Indeed, it is straighforward to verify that $t$ defines a morphism in $\Rep_{F_\ell}(\Ga_K)$ and that the $\Q_\ell$-dimensions of the source and the target agree. We are left to verify the injectivity. Let $f\in \Hom_{F_\ell}(V_\ell A,F_\ell)$ be non-zero, and let $x\in\Img(f)\setminus\{0\}$. By $F_\ell$-linearity, we may suppose that $x$ is the image of $1\in F_\l$ via $F_\l\hookrightarrow F_\ell$ for some $\l$. Then $\Tr_{F_\ell/\Q_\ell}(x)=[F_\l:\Q_\ell]\neq 0$, so $\Tr_{F_\ell/\Q_\ell}\circ f\neq0$. 

Composing $P'_{\Q_\ell}$ with $t$ gives an isomorphism $P'_{F_\ell}\colon \Hom_{F_\ell}(V_\ell A,F_\ell)\!\to\!V_\ell(A^\vee)(-1)$ in $\Rep_{F_\ell}(\Ga_K)$, which translates back to a $\Ga_K$-equivariant perfect $F_\ell$-bilinear pairing \[P_{F_\ell}\colon V_\ell A\times V_\ell(A^\vee)\to F_\ell(1).\]

Given a polarization $\l\colon A\to A^\vee$ defined over $K$, let $\cdot^\dag\colon\End^0_K(A)\to \End^0_K(A)$ be the corresponding Rosati involution. By composing the second argument of $P_{F_\ell}$ with $\l$, we obtain a map \[P^\l_{F_\ell}\colon V_\ell A\times V_\ell A\to F_\ell(1).\] 

\begin{prop}\label{prop:abelianV_real_det}
Let $A/K$ be an abelian variety with a polarization $\l$ such that $\End_K^0(A)$ contains a totally real field $F$ fixed by the Rosati involution $\cdot^\dag$. The pairing $P^\l_{F_\ell}$ is alternating, $F_\ell$-bilinear, $\Ga_K$-equivariant, and perfect. Consequently, $h:=\frac{2\dim A}{[F:\Q]}$ is an even integer, $\rho_\ell^*\cong\rho_\ell(1)$ in $\Rep_{F_\ell}(\Ga_K)$, and $\det_{F_\ell}\!\rho_\ell=\omega_K^{-h/2}$.
\end{prop}

\begin{proof}
By construction, the map $P^\l_{F_\ell}$ is $\Ga_K$-equivariant, $\Q_\ell$-bilinear, non-degenerate, and $F_\ell$-linear in the first variable. From the standard theory, for any $f\in\End_K(A)\otimes_\Z\Q_\ell$ and $a,b\in V_\ell A$ we have $P_{\Q_\ell}(a,\l(a))=0$ and $P_{\Q_\ell}(f(a),\l(b))=P_{\Q_\ell}(a,\l(f^\dag(b)))$. Using the isomorphism $t$ we see that the same statements are true when we replace $P_{\Q_\ell}$ by $P_{F_\ell}$. Since $\cdot^\dag$ acts trivially on $F_\ell$, the pairing $P_{F_\ell}^\l$ is $F_\ell$-bilinear and perfect. 

From classical bilinear algebra, the existence of an alternating and perfect $F_\ell$-bilinear pairing on $V_\ell A$ implies that $h$ is even. The $\Ga_K$-equivariance and non-degeneracy of $P_{F_\ell}^\l$ gives $(V_\ell A)^*\cong (V_\ell A)(-1)$ in $\Rep_{F_\ell}(\Ga_K)$. Taking the dual objects we obtain $\rho_\ell^*\cong \rho_\ell(1)$.  It remains to compute the determinant. 

For any $n\geq1$, the $F_\ell$-module $(\wedge^n(V_\ell A))^*$ of $F_\ell$-multilinear alternating $n$-forms on $V_\ell A$ is free of rank $\binom{h}{n}$. It is equipped with a natural $F_\ell$-linear Galois action given by $\s P=P\circ \wedge^n\rho^*_\ell(\s^{-1})$ for every $P\in(\wedge^n(V_\ell A))^*$ and every $\s\in\Ga_K$ (note that we need to put $\s^{-1}$ in order to have $(\tau\s)P=\tau(\s P)$). Then, the pairing $P_{F_\ell}^\l\in (\wedge^2(V_\ell A))^*$ satisfies $\s P_{F_\ell}^\l=\omega_K(\s^{-1})P_{F_\ell}^\l$. Using the notion of alternating product of multilinear forms, see \autocite[A~III.142, Exemple~3)]{bourbaki_alg1}, we consider the $h/2$-fold product $\phi=\left(P_{F_\ell}^\l\right)^{\wedge h/2}\in(\wedge^h(V_\ell A))^*$, on which $\s$ acts as multiplication by $\omega_K(\s)^{-h/2}$. Since $\Ga_K$ acts on the $F_\ell$-module $(\wedge^h(V_\ell A))^*=(\det_{F_\ell}V_\ell A)^*$ as $(\det_{F_\ell}\rho_\ell^*)^*\cong \det_{F_\ell}\rho_\ell$, we are left to check that $\phi\neq 0$. There is a suitable $F_\ell$-basis $e_1,\ldots,e_h$ of $V_{\ell}A$ such that the matrix of $P^\l_{F_\ell}$ in this basis is diagonal by blocks $\left( \begin{smallmatrix} 0&1\\ -1&0 \end{smallmatrix} \right)$. If we denote by $e_1',\ldots, e_h'$ the corresponding dual basis of $(V_\ell A)^*$, then $P_{F_\ell}^\l=\sum^{h/2}_{i=1}e_{2i-1}'\wedge e_{2i}'$. It is straightforward to verify that $\phi=(h/2)!\, e_1'\wedge\ldots\wedge e_{h}'$, and the latter is clearly non-zero. 
\end{proof}

\subsection{A geometric trichotomy}\label{subs:not_pot_good} 
Let $K/\Q_p$ be a finite extension, $\ell\neq p$ a prime, and $A/K$ an abelian variety. We consider the identity component $\Ac_{k_K}^0$ of the special fiber of the Néron model of $A/K$. By Barsotti--Chevalley theorem (see, e.g., \autocite[Thm.~8.27 and Prop.~16.15]{milneAG_book}), there exist a unipotent group $U/k_K$, a torus $T'/k_K$, and an abelian variety $B'/k_K$ that fit an exact sequence 
\begin{equation}\label{eq:barsotti_chevalley_seq}0\to T'\times U\to \Ac_{k_K}^0\to B'\to 0.\end{equation}
The functoriality of the Néron model and the facts that $\Hom_{k_K}(T',B')=0$, $\Hom_{k_K}(T',U)=0$, and $\Hom_{k_K}(U,B')=0$ imply that we have a morphism of unitary $\Q$-algebras \begin{equation}\label{eq:end_trichotomy}\End^0_K(A)\to \End^0_{k_K}(T')\times\End^0_{k_K}(B').\end{equation} The following proposition is a slightly sharper version of \autocite[Prop. 3.6.1]{ribet_real}. 

\begin{prop}\label{prop:geom_trichotomy}
We suppose that $A/K$ has real multiplication by $F$. Then, exactly one of the algebraic groups $T'$, $U$, and $B'$ is nontrivial. 
\end{prop}

\begin{proof}
Let us suppose that $T'$ is nontrivial.  From \eqref{eq:end_trichotomy} we get an inclusion $F\hookrightarrow \End_{k_K}^0(T')$. It follows that $X(T')\otimes_{\Z}\Q$ is a nontrivial $F$-vector space. Then, $\dim T'\geq [F:\Q]=\dim A=\dim\Ac_{k_K}^0$, so $U$ and $B'$ are trivial. 

It remains to prove that if $B'$ is nontrivial, then $\dim B'=\dim A$. We see from \eqref{eq:end_trichotomy} that $\End_{k_K}^0(B')$ contains $F$ as a subfield. Using Poincaré's complete reducibility theorem we may assume that $B'$ is some $n$-th power of a simple abelian variety $S'$. Then, using Albert's classification of endomorphism algebras of simple abelian varieties one can show that $[F:\Q]\leq n\dim S'=\dim B'$, this is done in \autocite[Lemma~6]{chai_rm_end}.
\end{proof}


\section{Rationality of representations on Tate modules}\label{sect:real_AV_potG_red}

\subsection{The setup}\label{subs:setup_pot_good} We fix a prime number $p$, a finite extension $K/\Q_p$, and an abelian variety $A/K$ having potentially good reduction. Let us suppose that $A/K$ has real multiplication (RM), i.e., that the $\Q$-algebra $\End_K^0(A)$ contains a totally real number field $F$ of degree $g=\dim A$ fixed by a Rosati involution on $\End_K^0(A)$. We recall the decomposition $F_\ell=F\otimes_\Q\Q_\ell=\prod_{\l}F_\l$. In addition, we have \begin{equation}\label{eq:F_ell_bar_decomposition} F_\C:=F\otimes_\Q\C=\prod_{\iota\colon F\hookrightarrow\C}\C.\end{equation}

Let $\rho_\ell$ be the Galois representation as in \ref{subs:l_adic_rep_av} for some $\ell\neq p$. It is also $F_\ell$-linear. As in \ref{subs:complex_WD}, let us fix an embedding $i:\Q_\ell\hookrightarrow \C$, then the associated complex Weil--Deligne representation is given by the pair $(W_i(\rho_\ell),0)$ where $W_i(\rho_\ell)$ denotes the $F_\C$-linear Weil representation obtained by restricting $\rho_\ell\otimes_{i,\Q_\ell}\C$ to $W_K$.

\begin{prop}\label{prop:F_mod_semis}
The representation $\rho_\ell$ is semisimple in $\Rep_{F_\ell}(W_K)$.
\end{prop} 

\begin{proof} 
It suffices to prove semisimplicity of the restriction to a subgroup $W_L\subseteq W_K$ of finite index (after applying the argument of \autocite[\S2.7 Lemma]{bh_llc}). Therefore, we may suppose that the reduction is already good over $K$. In this case, the inertia acts trivially and the action of the arithmetic Frobenius $\varphi_K^{-1}$ on $V_\ell A\cong V_\ell\Ac_{k_K}$ is induced by the Frobenius endomorphism $\phi$ of the reduced abelian variety $\Ac_{k_K}/k_K$. It is well known that $\Q[\phi]$ is a semisimple $\Q$-subalgebra of $\End^0_{k_K}(\Ac_{k_K})$. It follows that $F_\ell\otimes_{\Q}\Q[\phi]$ is a semisimple $\Q_\ell$-algebra, and we conclude that $\rho_\ell$ is semisimple. 
\end{proof}

\begin{prop} \label{prop:F_rationality}
For any $\s\in W_K$, let $P_{\ell,\s}\in F_\ell[T]$ be the $F_\ell$-characteristic polynomial\footnote{It is well-defined because of Thm.~\ref{thm:F_mod_free}.} of $\rho_\ell(\s)$. Then, $P_{\ell,\s}$ has coefficients in $F$ and each $F_\l$\mydash{}characteristic polynomial $P_{\l,\s}$ of $\rho_\l(\s)$ is the image of $P_{\ell,\s}$ via the inclusion $F[T]\hookrightarrow F_\l[T]$. 
\end{prop}

\begin{proof}
We will prove that $V_\ell(A)^*$ is an $F_\ell[\s]$-module that can be realized over $F$, i.e., there exists an $F[\s]$-module $W$ such that $W\otimes_{\Q}\Q_\ell\simeq V_\ell(A)^*$. 

The $F_\ell$-linear representation $\rho_{\ell}$ induces a morphism of $F_\ell$-algebras \[\nu\colon F_\ell[\s]\to\End_{F_\ell}(V_\ell (A)^*).\] As explained by \autocite[p.~499, Corollary]{serre_tate} and its proof, for every $\a\in F[\s]$ there exists an integer $n$ and finite extension $L/K$ over which $A$ has good reduction and such that the action of $n\a$ is induced by an endomorphism of the reduced abelian variety $\Ac_{k_L}$. Then, by a classical argument due to Weil (see \autocite[181]{mumford}), the $\Q_\ell$-characteristic polynomial of $\nu(\a)$ has rational coefficients.

The semisimplicity of $\rho_\ell$ implies that $F[\nu(\s)]$ is a semisimple $\Q$-algebra. Since $F[\nu(\s)]$ is also commutative and finite over $\Q$, we may write $F[\nu(\s)]=\prod_i\Q(\a_i)$ as a finite product of number fields and, accordingly, $(V_\ell A)^*=\prod_i V_i$. Each $V_i$ is the underlying space of a semisimple $\Q_\ell$-linear representation $\nu_i\colon \Q(\a_i)\to\End_{\Q_\ell}(V_i)$. By the above paragraph, for every $\a\in\Q(\a_i)$, the $\Q_\ell$-characteristic polynomial of $\nu_i(\a)$ has rational coefficients. Therefore, \autocite[p.~38, Lemma~1]{shimura_taniyama} applies 
and gives an isomorphism $V_i \simeq \Q(\a_i)^{d_i}\otimes_{\Q}\Q_\ell$ of semisimple $\Q_\ell(\a_i)$-modules for some positive integer $d_i$. 

Let us define the $F[\s]$-module $W:=\prod_i \Q(\a_i)^{d_i}$, so that, by construction, $W\otimes_{\Q}\Q_\ell\simeq V_\ell (A)^*$ as semisimple $F_\ell[\s]$-modules. Let $Q\in F[T]$ be the $F$-characteristic polynomial of $\s$ acting on $W$. Then, $P_{\ell,\s}=Q$. On the other hand, $P_{\ell,\s}$ can be calculated locally at each $\l$ and can be seen as a family of polynomials $(P_{\l,\s})_\l$ in $\prod_\l F_\l[T]$. Since $P_{\ell,\s}$ has coefficients in $F$, the family is constant.
\end{proof}

\begin{prop}\label{prop:rho_iota}
Let $A/K$ be an abelian variety with RM by $F$ and with potentially good reduction. Then we have the following decomposition of the complex Weil representation:
\begin{equation}\label{eq:V_ell(A)_bar_decomposition}
\W_i(\rho_\ell)=\prod_{\iota\colon F\hookrightarrow\C} \rho_\iota,\end{equation} 
where each $\rho_\iota$ is a semisimple Weil representation on a complex $2$\mydash{}dimensional vector space $V_\iota$. Furthermore:
\begin{enumerate}
\item The representations $\rho_\iota$ are $\Aut(\C)$-conjugate; i.e., for every pair of embeddings $\iota,\iota'\colon F\hookrightarrow \C$ there exists an automorphism $u\in\Aut\left(\C\right)$ such that $\rho_{\iota'}\cong\rho_{\iota}\otimes_{u}\C.$
\item Let $M/K$ be as in Def.~\ref{def:im}. The restrictions $\rho_\iota|_{I_K}$ have a common kernel $I_M$. Each $\rho_\iota$ thus induces a faithful $2$-dimensional representation of the finite group $I(M/K)$.
\item Each representation $\rho_\iota$ is essentially symplectic of weight $1$, i.e., $\rho_\iota(\frac{1}{2})$ is symplectic. In particular, $\rho_\iota^*\cong\rho_\iota(1)$ and $\det\rho_\iota=\omega_K^{-1}$. 
\item The root number $w(\rho_\iota,\psi_K)$ is $1$ or $-1$. 
\end{enumerate}
\end{prop}

\begin{proof}
By $F_{\C}$-linearity, the decomposition \eqref{eq:F_ell_bar_decomposition} implies \eqref{eq:V_ell(A)_bar_decomposition}, and we have $V_\iota\cong V_\ell(A)^*\otimes_{F_\ell,\iota}\C$ where the tensor product is taken over the unique extension $F_\ell\to\C$ of $\iota$ (having fixed $i\colon \Q_\ell\hookrightarrow\C$, see~\ref{subs:complex_WD}). Recall that $V_\ell(A)^*$ is a free $F_\ell$-module of dimension $2$ (Thm.~\ref{thm:F_mod_free}),  so $V_\iota$ is a complex Weil representation of dimension $2$. 

\begin{enumerate}[label={(\arabic*)}]
\item From Prop.~\ref{prop:F_rationality} we conclude that the $\C$-characteristic polynomial of $\s\in W_K$ acting on $V_\iota$ is the image via $\iota$ of a polynomial $P_{\s}\in F[T]$ that is independent of $\iota$. We recall that for any two embeddings $\iota,\iota'\colon F\hookrightarrow\C$, there exists a $u\in\Aut(\C)$ such that $\iota'=u\circ\iota$. The characteristic polynomials of $\rho_\iota(\s)\otimes_u\C$ and $\rho_{\iota'}(\s)$ are equal, so, by semisimplicity, $\rho_\iota\otimes_u\C\simeq \rho_{\iota'}$.  
\item Part (1) shows that all the $\rho_\iota$ have the same kernel, which must be the kernel of $\rho_\ell$. In particular, $\ker\rho_\iota|_{I_K}=\ker\rho_\ell \cap I_K=I_M$.
\item Applying the functor $-\otimes_{F_\ell,\iota}\C$ to the objects of Prop.~\ref{prop:abelianV_real_det} gives the result.
\item We use (3) together with Prop.~\ref{prop:root_n_props}.(1),(2) to get 
\begin{align*}w(\rho_\iota,\psi_K)^2&=w(\rho_\iota,\psi_K)w(\rho_\iota(1),\psi_K)\\ &=w(\rho_\iota,\psi_K)w(\rho_\iota^*,\psi_K)\\ &=\det(\rho_\iota(\theta_K(-1))).\end{align*} We know that $\det\rho_\iota=\omega_K^{-1}$ is unramified, so $w(\rho_\iota,\psi_K)^2=1$. \qedhere
\end{enumerate}
\end{proof}

\begin{rem}\label{rem:root_n_indep_iota}
Prop.~\ref{prop:rho_iota} shows that $(\rho_\iota)_\iota$ is a family of $\Aut(\C)$-conjugate representations of even dimensions that are essentially symplectic of odd weight. Then, \autocite[Thm.~1]{rohrlich_galois_inv} shows that $w(\rho_\iota,\psi_K)$ is independent of $\iota$. Let us fix some $\iota$. It follows from \eqref{eq:V_ell(A)_bar_decomposition} and multiplicativity of root numbers that \begin{equation}\label{eq:root_n_glob} w(A/K)=w(W_i(\rho_\ell),\psi_K)=
w(\rho_{\iota},\psi_K)^g.\end{equation} 
\end{rem}

\begin{prop}\label{prop:dichotomy}
We keep the hypotheses and notation of Prop.~\ref{prop:rho_iota}. Let us fix an embedding $\iota\colon F\hookrightarrow \C$ and regard $\rho_\iota$ as a representation of the group $G:=W(M/K)=I(M/K)\cdot\<\varphi_K\>$. Exactly one of the following is true:
\begin{enumerate}[label=\upshape(\alph*)]
\item $M/K$ is abelian. In this case $\rho_\iota=\chi_\iota\oplus\chi_\iota^{-1}\omega^{-1}_K$ for a character $\chi_\iota\colon G\to\C^\times$, which is faithful on $I(M/K)$, and there exists a finite totally ramified cyclic extension $L/K$ which is IM (see Def.~\ref{def:im}) for $A/K$; 
\item $M/K$ is non-abelian. In this case $\rho_\iota$ is irreducible. If $p\neq 2$, then $\rho_\iota=\Ind_H^G\chi_\iota$ where $\chi_\iota$ is a character of a normal subgroup $H\subset G$ of index $2$. Furthermore, every such $H$ is an abelian group containing $I^{\wild}(M/K)$ such that $H\cap I(M/K)$ is cyclic.
\end{enumerate}
\end{prop}

Before proving the proposition we establish a few lemmas. 

\begin{lem} \label{lem:faithful_on_G}
The representation $\rho_\iota$ is faithful on $G$.
\end{lem}

\begin{proof}
The eigenvalues of $\rho_\iota$ on $I(M/K)$ are roots of unity and the eigenvalues of $\varphi_K$ have absolute values $\sqrt{q_K}$. Since $\rho_\iota$ is faithful on $I(M/K)$, we infer that $\rho_\iota$ is faithful on $G$. 
\end{proof}

\begin{lem}\label{lem:comm}
The following statements are equivalent:
\begin{enumerate}
\item Every IM-extension $L'/K$ for $A/K$ is abelian; 
\item $M/K$ is abelian;
\item $\rho_\ell(\Ga_K)$ is abelian;
\item $\rho_\ell(W_K)$ is abelian;
\item $\rho_\iota(W_K)$ is abelian;
\item $\rho_\iota\colon W_K\to \GL_2(\C)$ is reducible;
\item $\rho_\iota$ is a direct sum of two characters of $W_K$;
\item $G$ is abelian.
\end{enumerate}
\end{lem}

\begin{proof}
The lemma is a slight generalisation of \autocite[Prop.~2.(ii)]{rohrlich_proof}. (1) implies that $M=L'K^{\unr}$ is abelian over $K$ as the compositum of two abelian extensions in $\overline{K}$, so we have (2). We have $(2)\Rightarrow(3)$ since $\rho_\ell$ factors through the quotient $\Gal(M/K)$ of $\Ga_K$. Restricting $\rho_\ell$ to $W_K$ gives $(3)\Rightarrow(4)$. The implication $(4)\Rightarrow(5)$ follows from the natural projection $\W_i(\rho_\ell)\to \rho_\iota$ (see \eqref{eq:V_ell(A)_bar_decomposition}). Schur's lemma gives $(5)\Rightarrow(6)$, and $(6)\Rightarrow (7)$ follows by semisimplicity of $\rho_\iota$. The statement (7) implies that $\rho_\iota(G)$ is abelian, and thus $G$ is abelian by Lemma~\ref{lem:faithful_on_G}, thus giving $(8)$. We prove $(8)\Rightarrow (2)$ by recalling that $G=W(M/K)$ is a dense subgroup of $\Gal(M/K)$, so commutativity of the former implies comutativity of the latter. We are left to establish $(2)\Rightarrow(1)$, which follows by recalling that every IM-extension for $A/K$ is a subextension of $M/K$. \end{proof} 

\begin{proof}[Proof of Prop.~\ref{prop:dichotomy}]

If $M/K$ is abelian, then $\rho_\iota$ decomposes into a sum of two characters (Lemma~\ref{lem:comm}) whose product is $\det\rho_\iota=\omega_K^{-1}$ (see Prop.~\ref{prop:rho_iota}.(3)), so we may write $\rho_\iota=\chi_\iota\oplus\chi_\iota^{-1}\omega_K^{-1}$. Since $\omega_K$ is unramified, the character $\chi_\iota$ sends $I(M/K)$ injectively to a finite subgroup of $\C^\times$, which must be cyclic. Let us fix a Frobenius lift $\varphi_K$, and let $L=L'$ be the extension given by Lemma~\ref{lem:min_ext_good}.(3), which is abelian by Lemma~\ref{lem:comm}. We see that $\Gal(L/K)=I(L/K)\simeq I(M/K)$ is cyclic.

Let us suppose that $M/K$ is non-abelian. Lemma~\ref{lem:comm} shows that $\rho_\iota$ is irreducible. If $p\neq 2$, then applying \autocite[\nopp (2.2.5.3)]{tate_background} shows that there exist a subgroup $H\subset G$ of index $2$ and a character $\chi_\iota$ of $H$ such that $\rho_\iota=\Ind_H^G\chi_\iota$. 

By adjunction, the restriction $\rho_\iota|_H$ contains $\chi_\iota$ as a subrepresentation, so $\rho_\iota|_H=\chi_\iota\oplus\chi_\iota^{-1}\omega_K^{-1}$, and thus $\rho_\iota|_H$ has abelian image. By faithfulness, $H$ is abelian, and $\chi_\iota$ identifies the finite group $I(M/K)\cap H$ with its cyclic image in $\C^\times$.  

If $H$ does not contain $I^{\wild}(M/K)$, then $H\cdot I^{\wild}(M/K)=G$ and \[2=\left|G/H\right|= \left|\bigslant{I^{\wild}(M/K)}{I^{\wild}(M/K)\cap H}\right|.\] Therefore, 2 must be a power of $p$, which is impossible if $p\neq 2$. 
\end{proof}

\begin{cor}\label{cor:cyclic_wild_inertia}
The group $I^{\wild}(M/K)$ is cyclic if $M/K$ is abelian or if $p\neq 2$. 
\end{cor}

\begin{proof}
If $M/K$ is abelian, then every subgroup of $I(M/K)$ is cyclic by Prop.~\ref{prop:dichotomy}.(a). If $M/K$ is non-abelian and $p\neq 2$, then $I^{\wild}(M/K)$ is contained in the cyclic group $H\cap I(M/K)$.
\end{proof}

\begin{rem}
We note that the subgroup $H$ in the case (b) is not unique for $\iota$. We will show how to make a more precise choice of $H$ in Lemma~\ref{lem:H_unramified}.\end{rem}

\subsection*{Other restrictions} 

\begin{lem}\label{lem:phi(e)_and_2g}
Let $A/K$ be an abelian variety of dimension $g$ with RM and potentially good reduction. Let $\s\in \rho_\ell(I_K)$ be an element of order $d$. Denote by $\varphi$ Euler's totient function. Then $\varphi(d)$ divides $2g$. 
\end{lem}

\begin{proof}
Recall from Prop.~\ref{prop:F_rationality} that the $F_\ell$-characteristic polynomial $P_\s$ of $\s$ is of degree $2$ and has coefficients in $F$. Since $\s^d=\id$, the complex roots of $P_{\s}$ are among the roots of $X^d-1$. Since $\det \s=1$ (by Prop.~\ref{prop:rho_iota}.(3)), the two roots of $P_{\s}$ are roots of unity $\zeta$ and $\zeta^{-1}$ of some order $d'|d$. Then $\s^{d'}=\id$, which implies that $d'=d$. On the other hand, $\a:=\Tr(\s)=\zeta+\zeta^{-1}\in F$. Since $\a$ has degree $\max\{1,\frac{\varphi(d)}{2}\}$ over $\Q$, we conclude that $\varphi(d)\big| 2 g$.
\end{proof}

As we have seen in Prop.~\ref{prop:rho_iota}.(2), the representations $\rho_\iota|_{I_K}$ induce faithful $\Aut(\C)$-conjugate representations of a finite group $I(M/K)$. Let $p^re=|I(M/K)|$ with $e$ prime to $p$. If $p\neq 2$, then Cor.~\ref{cor:cyclic_wild_inertia} shows that $I^{\wild}(M/K)$ is cyclic. 

\begin{prop}\label{prop:odd_dim_rootn}
For $A/K$ as in Lemma~\ref{lem:phi(e)_and_2g}, the following statements hold:
\begin{enumerate}
\item If $g$ is even, then $w(A/K)=1$;
\item If $r\geq1$ and $p\neq2$, then $p^{r-1}(p-1)|2g$; in particular, if $g$ is odd, then $p\equiv3\bmod 4$;  
\item If $g$ is odd, then $e$ can only be $s^m$, $2s^m$, or $4$, where $m\geq0$ and $s\equiv3\bmod4$ is a prime different from $p$.
\end{enumerate}
\end{prop}

\begin{proof}
(1) follows from Prop.~\ref{prop:rho_iota}.(4) and \eqref{eq:root_n_glob}.
Applying Lemma~\ref{lem:phi(e)_and_2g} to a generator of $I^{\wild}(M/K)$ and using $\varphi(p^r)=p^{r-1}(p-1)$ we obtain (2). 
Suppose now that $g$ is odd. If $\varphi(e)$ is odd, then $e$ is $1$ or $2$. If $\varphi(e)$ is even, then applying Lemma~\ref{lem:phi(e)_and_2g} to a generator of $I^{\tame}(M/K)$ gives $\varphi(e)\equiv 2\bmod 4$. Now (3) follows from the usual formulas of $\varphi(e)$.  
\end{proof}

\section{The case of abelian inertia}\label{sect:ab_inertia}

We continue to work in the setting of \ref{subs:setup_pot_good} and suppose that $p\neq2$. We adopt the notation of Prop.~\ref{prop:rho_iota} and regard each $\rho_\iota$ as a faithful representation of $G=W(M/K)$ (see Lemma~\ref{lem:faithful_on_G}). Let us write $|I(M/K)|=p^re$ with $e=|I^{\tame}(M/K)|$, so that $p\nmid e$. Applying \eqref{eq:root_n_glob}, it suffices to determine the root number $w(\rho_{\iota}, \psi_K)$ for a fixed embedding $\iota\colon F\hookrightarrow \C$.
\begin{thm}\label{thm:av_root_N:abelian}
If $\rho_\ell(\Ga_K)$ is abelian, then $e\mid (q_K-1)$ and \[w(\rho_{\iota},\psi_K)=(-1)^{\frac{q_K-1}{e}}.\]
\end{thm}

\begin{proof} By Lemma~\ref{lem:comm}, the image $\rho_\ell(\Ga_K)$ is abelian if and only if $M/K$ is abelian, so we are in the case (a) of Prop.~\ref{prop:dichotomy}. Then, we have a decomposition $\rho_{\iota}=\chi_{\iota}\oplus\chi_{\iota}^{-1}\omega_K^{-1}$. Using multiplicativity of the root number and Prop.~\ref{prop:root_n_props}.(1),(2) we obtain \[w(\rho_{\iota},\psi_K)=\chi_{\iota}(\theta_K(-1))\in\{\pm1\}.\] 

By Lemma~\ref{lem:min_ext_good}, there exists a finite IM-extension $L/K$ for $A/K$. By Lemma~\ref{lem:comm}, $L/K$ is abelian. We identify $I(L/K)\simeq I(M/K)$. Then the composition $(\chi_{\iota}\circ\theta_K)|_{\O_K^\times}$ can be seen as the following sequence of group homomorphisms: 

\begin{equation}\label{eq:finite_recipr}\begin{tikzcd}\
\O_K^\times \arrow[r, two heads] & \bigslant{\O_K^\times}{\norm_{L/K}(\O_{L}^\times)} \arrow[r,"\theta_{L/K}" , "\widetilde{\hspace{0.5cm}}"'] &[0.1cm] I^{\wild}(L/K)\times I^{\tame}(L/K) \arrow[hookrightarrow]{r}{\chi_{\iota}} & \C^\times.
\end{tikzcd}\end{equation}

Let us recall the identification $\O_K^\times\cong k_K^{\times}\times(1+\m_K)$. The image of $k_K^\times$ is trivial in $I^{\wild}(L/K)$, so $\theta_{L/K}$ induces a homomorphism $\theta_{k_K}\colon k_K^\times\to I^{\tame}(L/K)$. On the other hand, the image of the pro-$p$ group $1+\m_K$ is trivial in $I^{\tame}(L/K)$, so $\theta_{k_K}$ must be surjective. In particular, $e\mid(q_K-1)$. Since $p\neq2$, the class of $-1$ in $k_K^\times$ is nontrivial, so $\chi_{\iota}(\theta_K(-1))=\chi_{\iota}(\theta_{k_K}(-1))=1$ if and only if $-1$ belongs to the unique subgroup of index $e$ in $k_K^\times$, which can be characterized by $\{x\in k_K^\times:x^{\frac{q_K-1}{e}}=1\}$.
\end{proof}

\subsection{Representation having non-abelian image}\label{subsect:tame_reduction_calc} It remains to study the case where $\rho_\ell(\Ga_K)$ is non-abelian, which is the case (b) of Prop.~\ref{prop:dichotomy} (see Lemma~\ref{lem:comm}), when $M/K$ is non-abelian. In this case, $\rho_{\iota}$ is induced by a character $\chi_{\iota}$ of an abelian normal subgroup $H\subset G=W(M/K)$ of index 2, which contains $I^{\wild}(M/K)$. 

\begin{lem}\label{lem:frob_center}
We can choose a geometric Frobenius lift $\varphi_K$ so that $\varphi_K^2$ is contained in the center of $G$.
\end{lem}

\begin{proof}
Let $\tau_K$ and $\varphi_K$ be as in Thm.~\ref{thm:galois_split_hard}. Since $H$ is commutative of index 2 in $G$ and contains $I^{\wild}(M/K)$, the element $\varphi^2_K$ is in $H$ and commutes with every element of $I^{\wild}(M/K)$. We are left to prove that $\varphi^{-2}_K\tau_K\varphi^2_K=\tau_K$, which is equivalent to $e$ dividing $q_K^2-1$ by the aforementioned theorem. If $\tau_K\in H$, there is nothing to prove since $H$ is abelian, so we may suppose that $\tau_K\not\in H$. Then, $\tau_K^2\in H$ and $e=|I^{\tame}(M/K)|$ is even. 

If $\varphi_K\in H$, then $\tau_K^{2q_K}=\varphi^{-1}_K\tau_K^2\varphi_K=\tau_K^2$, so $e|2q_K-2$, which implies that $e|q_K^2-1$. 

If $\varphi_K\not\in H$, then $\varphi_K\tau_K \in H$ commutes with $\varphi^2_K$, so $\tau_K=\varphi^{-3}_K(\varphi_K\tau_K)\varphi^2_K$, and we are done. \end{proof}

\begin{hyp}\label{hyp}
The image of the inertia subgroup $I_K$ via the representation $\rho_\ell$ is commutative. 
\end{hyp}

The hypothesis is verified in the following cases: 
\begin{enumerate} 
\item $\rho_\ell|_{I_K}$ factors through the tame inertia group $I^{\tame}_K$, or, equivalently, $A/K$ attains good reduction over a finite (at most) tamely ramified extension $L'/K$, the explicit formulas for root numbers are given by \autocite[Thm.~1.4]{bisatt}; 
\item $A/K$ is an elliptic curve with discriminant of even valuation, this case is settled in \autocite[5.2.~b)]{kobayashi};
\item $A/K$ has complex multiplication, see \autocite[p.~502, Cor.~2]{serre_tate}.
\end{enumerate}

By Prop.~\ref{prop:rho_iota}.(1), the group $\rho_\ell(I_K)$ is commutative if and only if $\rho_\iota(I_K)$ is commutative for any $\iota$.

\begin{lem}\label{lem:H_unramified}
If Hypothesis~\ref{hyp} is satisfied, then we can choose $H$ so that, independently of $\iota$, the representation $\rho_{\iota}$ is induced by a character of $H$ and that the extension $M^{\overline{H}}/K$ is unramified (here $\overline{H}\subset \Gal(M/K)$ denotes the closure of $H$). \end{lem}

\begin{proof}
Let us set $H:=I(M/K)\times \<\varphi_K^2\>$ for a lift $\varphi_K$ as in Lemma~\ref{lem:frob_center}, so that $H$ identifies with a commutative subgroup of $G$ of index $2$, independent of $\iota$. By semisimplicity, the restriction $\rho_{\iota}|_{H}$ decomposes to a sum of two characters 
\[\rho_{\iota}|_{H}=\chi_{\iota}\oplus\chi_{\iota}^{-1}\omega_K^{-1}\] and then by the adjunction property we have a nontrivial morphism of complex $G$-representations \[\Ind_{H}^G\chi_{\iota}\to \rho_{\iota},\] which is surjective as $\rho_{\iota}$ is irreducible (by Lemma~\ref{lem:comm}) and therefore an isomorphism since the dimensions agree. On the other hand, if $L_u/K$ is the unramified quadratic extension, then $W(M/L_u)$ is a subgroup of $G$ of index $2$ and contains $I(M/K)$. The element $\varphi_K^2$ is a lift of $\Frob^{-1}_{k_{L_u}}$ in $\Ga_{L_u}$, so $\varphi_K^2\in W(M/L_u)$. Therefore, $H\subseteq W(M/L_u)$. The latter inequality is an equality because both subgroups have index $2$ in $G$, thus $\overline{H}=\Gal(M/L_u)$. \end{proof}

\begin{thm}\label{thm:non_ab_commutative_inertia}
We suppose that $\rho_\ell(\Ga_K)$ is non-abelian and that $\rho_\ell(I_K)$ is abelian. For the subgroup $H\subset G$ from Lemma~\ref{lem:H_unramified}, we have $\rho_\iota=\Ind_H^G\chi_\iota$. Then:
\begin{enumerate}
\item $a(\rho_\iota)=2\cdot a(\chi_\iota)$, where $a(\cdot)$ denote the Artin conductor, 
\item $e$ divides $q_K+1$, and
\item $w(\rho_\iota,\psi_K)=(-1)^{\frac{a(\rho_\iota)}{2}+\frac{q_K+1}{e}}.$
\end{enumerate}
\end{thm}

\begin{proof}
Let us denote $L_u:=M^{\overline{H}}$, which is quadratic and unramified over $K$ by Lemma~\ref{lem:H_unramified} and its proof. (1) follows from the general formulas of Artin conductors (see, e.g.,  \autocite[\S10]{rohrlich}). 

The formula \ref{subs:eps_factor}.(ii) for the $\e$-factor of an induced representation gives
\begin{equation}\label{eq:eps_ind_formula}w(\rho_\iota,\psi_K)=w(\chi_\iota,\psi_{L_u})\,w\!\left(\Ind_H^G\1_H,\psi_K\!\right)\!,\end{equation} 
since $w(\1_H, \psi_{L_u})^{-1}=1$ (from \ref{subs:eps_factor}.(iii)). 

Let $\chi_0$ be the unramified quadratic character of $G$ given by the composition $G\to\Gal(L_u/K)\cong\{-1,1\}$. Then $\Ind_H^G\1_H\cong \1_G\oplus\chi_0$. Using \ref{subs:eps_factor}.(i),(iii) we have: \begin{equation}\label{eq:eps_ind_calc}w(\Ind_H^G\1_H,\psi_K)=w(\1_G,\psi_K)w(\chi_0,\psi_K)=\chi_0(\theta_K(c))=(-1)^{n(\psi_K)},\end{equation} where $c\in K^{\times}$ has valuation $n(\psi_K)$.

We prove (2) and compute $w(\chi_\iota,\psi_{L_u})$ in the following lemma.

\begin{lem}\label{lem:eps_factor_chi}
We have $e\mid(q_K+1)$ and 
\begin{equation}\label{eq:chi_iota_calc} w(\chi_\iota,\psi_{L_u})=(-1)^{n(\psi_K)+a(\chi_\iota)+\frac{q_k+1}{e}}.\end{equation}
\end{lem}

\begin{proof}
Let $\chi_0$ be the character corresponding to $L_u/K$ as before, and let $t\colon G^{\ab}\to H^{\ab}=H$ be the transfer homomorphism, which corresponds to the inclusion $K^\times\hookrightarrow L_u^{\times}$ via the reciprocity maps. The twisted representation $\rho_\iota(\frac{1}{2})$ has trivial determinant (see Prop.~\ref{prop:rho_iota}.(3)), so Deligne's determinant formula from \autocite[508]{deligne_eq_fonctionelle} gives 
\begin{equation}\label{eq:det_induced} 1=\det\left(\Ind_H^G\chi_\iota(\textstyle{\frac{1}{2}})\right)(g)=\chi_0(g)\chi_\iota(\textstyle{\frac{1}{2}})(t(g))\end{equation} 
for every $g\in G$. Therefore, for every $x\in K^\times,$
\begin{equation}\label{eq:restriction_formula}(\chi_\iota(\textstyle{\frac{1}{2}})\circ\theta_{L_u})(x)=\chi_0^{-1}\left(\theta_{K}(x)\right)=(-1)^{v_K(x)}.\end{equation} 

Let $\chi_1$ be the nontrivial quadratic unramified character of $H$, so that for all $x\in L_u$ we have ${\chi_1\circ\theta_{L_u}(x)}=(-1)^{v_{L_u}(x)}$. Since the valuations $v_K$ and $v_{L_u}$ agree on $K^\times$, the character $\left(\chi_\iota(\textstyle{\frac{1}{2}})\cdot \chi_1\right)\circ \theta_{L_u}$ is trivial on $K^\times$. The extension $L_u/K$ is unramified and $[L_u:K]=2$, so $L_u=K(\zeta_{2q_K-2})$ where $\zeta_{2q_K-2}$ is a primitive root of unity of order $2q_K-2$. We have $\zeta_{2q_K-2}^2\in K^\times$, so applying \autocite[Thm.~3]{frohlich_queyrut} gives 
\begin{align}\label{eq:char_prod_epsilon} 
\nonumber w(\chi_\iota(\textstyle{\frac{1}{2}})\cdot \chi_1,\psi_{L_u})&=\chi_\iota(\textstyle{\frac{1}{2}})\left(\theta_{L_u}(\zeta_{2q_K-2})\right)\cdot \chi_1\left(\theta_{L_u}(\zeta_{2q_K-2})\right)\\
&=\chi_\iota\left(\theta_{L_u}(\zeta_{2q_K-2})\right),
\end{align} 
the last equality holds because $\chi_1$ and $\omega_K^{1/2}$ are unramified. On the other hand, applying Prop.~\ref{prop:root_n_props}.(1) gives
\begin{equation}\label{eq:chi_iota_prod_eps} w(\chi_\iota(\textstyle{\frac{1}{2}})\cdot \chi_1,\psi_{L_u})=w(\chi_\iota,\psi_{L_u})\cdot (-1)^{n(\psi_{L_u})+a(\chi_\iota)}. \end{equation}
Combining \eqref{eq:char_prod_epsilon}, \eqref{eq:chi_iota_prod_eps}, and the observation that $n(\psi_K)=n(\psi_{L_u})$, we obtain 
\begin{equation}\label{eq:chi_eps_zeta}w(\chi_\iota,\psi_{L_u})=(-1)^{n(\psi_{K})+a(\chi_\iota)}\chi_\iota(\theta_{L_u}(\zeta_{2q_K-2})).\end{equation} 

Let $L/K$ be a finite, Galois, and IM extension for $A/K$ (see Lemma~\ref{lem:min_ext_good}). Then $L_uL/K$ is also a Galois and IM extension, so we identify $I(M/K)\simeq I(L/K)\simeq I(L_uL/L_u)$. As in the proof of Thm.~\ref{thm:av_root_N:abelian}, by using the decomposition $\O_{L_u}^\times\cong k_{L_u}^\times\times (1+\m_{L_u})$ we obtain a surjective ho\-mo\-mor\-phism $\theta_{k_{L_u}}\!:\!k_{L_u}^\times\!\twoheadrightarrow \!I^{\tame}(L_uL/L_u)$ induced by $\theta_{L_u}$. We see from \eqref{eq:restriction_formula} that $\chi_\iota\circ \theta_{L_u}$ is trivial on $\O_K^\times$. It follows that the subgroup $\ker(\chi_\iota\circ\theta_{L_u}|_{k_{L_u}^\times})=\ker(\theta_{k_{L_u}})$ contains $k_K^\times$. This implies that $e=|I^{\tame}(L_uL/L_u)|$ divides $[k_{L_u}^\times:k_K^\times]=q_K+1$. 

The subgroup $\ker(\chi_\iota\circ\theta_{L_u}|_{k_{L_u}^\times})$ 
of index $e$ in $k_{L_u}^\times$ contains $\zeta_{2q_K-2}$ if and only if $1=\zeta_{2q_K-2}^{(q_K^2-1)/e}=(-1)^{\frac{q_K+1}{e}}$. Since $(\chi(\frac{1}{2})\cdot \chi_1)\circ\theta_{L_u}|_{K^\times}$ is trivial, \eqref{eq:chi_eps_zeta} gives $\chi(\theta_{L_u}(\zeta_{2q_K-2}))^2=1$, and thus \eqref{eq:chi_iota_calc} follows.
\end{proof}
Plugging \eqref{eq:eps_ind_calc} and \eqref{eq:chi_iota_calc} into \eqref{eq:eps_ind_formula}, as well as using (1), we obtain (3).   \end{proof}

\begin{cor}\label{cor:exposition_pot_good}
Let $A/K$ be as in \ref{subs:setup_pot_good}, let denote by $a(A/K)$ its Artin conductor, and let $e$ be the largest prime-to-$p$ divisor of the order of $\rho_\ell(I_K)$. 
\begin{enumerate}
\item If $\rho_\ell(\Ga_K)$ is commutative, then $w(A/K)=(-1)^{\frac{g(q_K-1)}{e}};$
\item If $\rho_\ell(\Ga_K)$ is non-commutative and $\rho_\ell(I_K)$ is commutative, then \[w(A/K)=(-1)^{\frac{a(A/K)}{2}+\frac{g(q_K+1)}{e}}.\]
\end{enumerate}
\end{cor}

\begin{proof}
Let us recall formula \eqref{eq:root_n_glob}. Then, (1) follows from Thm.~\ref{thm:av_root_N:abelian}; the part (2) follows from Thm.~\ref{thm:non_ab_commutative_inertia}.(3) and multiplicativity of Artin conductor.
\end{proof}

\begin{rem}
If $\rho_\ell(I^{\wild}_K)$ is trivial, then Corollary~\ref{cor:exposition_pot_good} allows us to compute the root number $w(A/K)$. In this case, if $A/K$ has bad potentially good reduction, then we always have $a(A/K)=2g$. The obtained formulas are special cases of the results of \autocite[Thm.~1.4]{bisatt}, as it can be verified by a calculation using Prop.~\ref{prop:odd_dim_rootn}.
\end{rem}

\section{Potentially totally toric reduction}\label{sect:real_AV_pot_mult}

Let $A/K$ be an abelian variety of dimension $g$ with RM by $F$, and recall the notation and results of \ref{subs:p_adic_unif}. We suppose that $A/K$ does not have potentially good reduction. Then it must have potentially totally toric reduction by Prop.~\ref{prop:geom_trichotomy}, so $B=0$. The analytification $A/K$ is a quotient of the analytification of a torus $T/K$, which gives rise to a $\Q$-linear Weil representation \[\eta\colon W_K\to\GL_{\Q}(X(T)\otimes_{\Z}\Q)\] with finite image. Prop.~\ref{prop:sabitova} then gives 
\begin{equation}\label{eq:pot_mult_repr}\WD_i(\rho_\ell)\cong \eta(-1)\otimes\spe(2).\end{equation}
Let us fix a basis $v_1,\ldots,v_g$ of the underlying vector space of $\eta$ and let $e_0,e_1$ be the standard basis of $\spe(2)=\left(\1\oplus \omega_K, \left( \begin{smallmatrix} 0&0\\ 1&0 \end{smallmatrix} \right)\right)$. Then, in the basis \[\mathbb{B}:=(v_1\otimes e_0, \ldots,v_g\otimes e_0,v_1\otimes e_1,\ldots,v_g\otimes e_1),\] the WD-representation $\WD_i(\rho_\ell)$ is given as \begin{equation}\label{eq:pot_tor_WD} \WD_i(\rho_\ell)\simeq\left(\eta(-1)\oplus\eta, \left(\!\begin{smallarray}{c|c}0_g  & 0_g \\ \hline I_g & 0_g\end{smallarray}\!\right)\!\right).\end{equation}
From \eqref{eq:pot_tor_WD} one can recover the isomorphism class of the representation $\rho_\ell$ by a procedure inverse to the one in \ref{subs:l_adic_monodromy}. Let us choose a continuous nontrivial homomorphism $t_\ell\colon I_K\to\Q_\ell$. Then, in the same basis $\mathbb{B}$, for every $j\in I_K$, 
\begin{equation} \label{eq:pot_tor_rep_inertia} 
\rho_\ell(j)=\left(\!\begin{smallarray}{c|c}\eta(j)  & 0_g \\ \hline 0_g & \eta(j)\end{smallarray}\!\right)\cdot\exp\!{\left(\!\begin{smallarray}{c|c}0_g  & 0_g \\ \hline t_\ell(j)I_g & 0_g\end{smallarray}\!\right)}=\left(\!\begin{smallarray}{c|c}\eta(j)  & 0_g \\ \hline t_\ell(j)\eta(j) & \eta(j)\end{smallarray}\!\right),
\end{equation} and 
\begin{equation}\label{eq:pot_tor_rep_frob} \rho_\ell(\varphi_K)=\left(\!\begin{smallarray}{c|c}q_K\eta(\varphi_K)  & 0_g \\ \hline 0_g & \eta(\varphi_K)\end{smallarray}\!\right). \end{equation}

\begin{prop}\label{prop:pot_tor_char} An abelian variety $A/K$ with RM which does not have potentially good reduction must have one of the following reduction types: 
\begin{enumerate}[label=\upshape(\alph*)]
\item $\Ac_{k_K}^0$ is a split torus, or, equivalently, $\eta$ is the trivial representation (split multiplicative reduction);
\item $\Ac_{k_K}^0$ is a non-split torus, or, equivalently, $\eta$ is unramified and nontrivial (non-split multiplicative reduction);
\item $\Ac_{k_K}^0$ is unipotent, but $A/L$ has completely toric reduction for some finite $L/K$, or, equivalently, $\eta$ is ramified (additive potentially multiplicative reduction).
\end{enumerate}
\end{prop}

\begin{proof} Let us suppose that $\eta$ is unramified. Then $\rho_\ell|_{I_K}$ is unipotent, as it can be seen from \eqref{eq:pot_tor_rep_inertia}, so $A/K$ has semistable reduction by Thm.~\ref{thm:ab_var_reduction}.(2), and thus $\Ac_{k_K}^0$ must be a torus $T'$ (see Prop.~\ref{prop:geom_trichotomy}). After writing the matrices of the dual representation $\rho_\ell^*$ in the dual basis of $\mathbb{B}$, we see that \[(V_\ell A)^{I_K}=\Vect(v_1^*\otimes e_0^*,\ldots,v_g^*\otimes e_0^*),\] on which $\varphi_K$ acts as $q_K^{-1}\eta^*(\varphi_K)=\eta^*(1)(\varphi_K)$. On the other hand, we have isomorphisms of $\ell$-adic $\Ga_{k_K}$-representations  $(V_\ell A)^{I_K}\cong V_\ell \Ac_{k_K}$ (see \autocite[Lemma~2]{serre_tate}), $V_\ell\Ac_{k_K}\cong V_\ell T'$ (since $\Ac_{k_K}/\Ac^0_{k_K}$ is finite), and, formally, $V_\ell T'\cong \Hom(X(T')\otimes_{\Z}\Q_\ell,\Q_\ell)(1)$. Therefore, the $\Ga_{k_K}$-action on $X(T')$ is trivial if and only if $\eta^*$ is trivial. Since $\eta$ is $\Q$-linear, we have $\eta\simeq\eta^*$. Therefore, we conclude that $\eta$ is trivial if and only if $T'$ is split.

Suppose that $\eta$ is ramified. The image of $\eta|_{I_K}$ is finite, so $\eta|_{I_K}$ and, subsequently, $\rho_\ell|_{I_K}$ cannot be unipotent. Thm.~\ref{thm:ab_var_reduction}.(2) implies that $A/K$ does not have semistable reduction, so $\Ac_{k_K}^0$ is not a torus and we may conclude via Prop.~\ref{prop:geom_trichotomy}.
\end{proof}

\subsection{$F$-rationality again}\label{subs:pot_mult_F-rationality} Recalling the description of $\eta$ given by Remark~\ref{rem:p-adic_unif_funct}.(2) and using \eqref{eq:end_trichotomy} we see that $\eta$ is $F$-linear of dimension one. We may then regard $\eta$ as a homomorphism $\eta_F\colon W_K\to F^\times.$ We have a decomposition $F\otimes_{\Q}\C\cong \prod_{\iota\colon F\to\C}\C$, so 
\begin{equation} \eta\otimes_{\Q}\C \cong\prod_{\iota\colon F\to \C}\eta_\iota,\end{equation} with $\eta_\iota:=\eta_F\otimes_{F,\iota}\C$, where the structural morphism is given by $\iota$. Consequently, defining $\rho'_\iota:=\eta_\iota(-1)\otimes\spe(2)$, the isomorphism \eqref{eq:pot_mult_repr} gives \begin{equation}\label{eq:decomp_not_pot_good}\WD_i(\rho_\ell)\cong \prod_\iota \rho'_\iota.\end{equation}
We note that $\eta_\iota=\iota\circ\eta_F$, which implies that the $\eta_\iota$'s are $\Aut(\C)$-conjugate, and thus the $\rho'_\iota$'s are also $\Aut(\C)$-conjugate. As in Rem.~\ref{rem:root_n_indep_iota}, we may apply \autocite[Thm.~1]{rohrlich_galois_inv} to see that $w(\rho'_\iota,\psi_K)$ is independent of $\iota$.

\begin{thm}\label{thm:real_potMult_rootN_formulas} 
Let us fix some $\iota\colon F\hookrightarrow \C$. In the ongoing notation, 
\[w(\rho'_\iota,\psi_K)=\begin{cases}-1 & \text{if $A/K$ has split multiplicative reduction};\\ 
1 & \text{if $A/K$ has non-split multiplicative reduction};\\ 
(-1)^{\frac{q_K-1}{2}}&\parbox[t]{.62\textwidth}{if $A/K$ has additive potentially multiplicative \\ reduction and $p\neq2.$}\end{cases}\]
\end{thm}

\begin{proof} We observe that, since $F$ is totally real, $\eta_F$ and $\eta_\iota$ are quadratic. Applying Prop.~\ref{prop:root_n_props}.(1),(3) gives \[w(\rho'_\iota,\psi_K)=\eta_\iota\left(\theta_K(-1)\right)\cdot (-1)^{\<\eta_\iota|\1\>}.\] Since $\eta_\iota$'s are $\Aut(\C)$-conjugate, we may replace $\eta$ with any $\eta_\iota$ in Prop.~\ref{prop:pot_tor_char}. 

Let us suppose that $\eta_\iota$ is unramified. Then $\eta_\iota(\theta_K(-1))=1$. Depending on whether $\<\eta_\iota|\1\>$ is 1 or 0, the reduction is split or non-split multiplicative, respectively, since $\<\eta_\iota|\1\>=1$ if and only if $\eta_\iota$ is trivial. 

It remains to treat the case when $\eta_\iota$ is ramified. In particular, $\eta_\iota$ is non-trivial, so $\<\eta_\iota|\1\>=0$. More precisely, $\eta_\iota$ is of exact order $2$, so it factors through a quotient $\Gal(L/K)$ of order $2$. If $p\neq 2$, then $L/K$ is totally tamely ramified. In that case, $\eta_\iota(\theta_K(-1))=1$ if and only if $-1$ is a norm for $L/K$. The latter is equivalent to $-1$ being a square in $k_K$, which happens exactly when $q_K\equiv 1\bmod 4$.\end{proof}

\begin{cor}\label{cor:real_potMult_rootN_formulas} 
For $A/K$ given at the begining of Section~\ref{sect:real_AV_pot_mult} and for $p\neq 2$, 
\[w(A/K)=\begin{cases}(-1)^g & \text{if the reduction (over $K$) is split multiplicative};\\ 
1 & \text{if the reduction is non-split multiplicative};\\ (-1)^{g\frac{q_K-1}{2}}&\text{if the reduction is additive}.\end{cases}\]
\end{cor}

\begin{proof}
The formulas follow from \eqref{eq:decomp_not_pot_good}, multiplicativity of root numbers, and Thm~\ref{thm:real_potMult_rootN_formulas}.
\end{proof}

\subsection*{Acknowledgements} The present paper is based on the work I have done during my doctoral studies at the University of Strasbourg. I wholeheartedly thank my thesis advisors Adriano Marmora and Rutger Noot for proposing the research topic, for numerous hours of helpful discussions, as well as for proof-reading the draft of the manuscript. I thank Kęstutis Česnavičius and Takeshi Saito for pointing out some errors in the previous version. I also thank the anonymous referee for the comments and suggestions. 

\printbibliography

@book {ab_var_cmplx,
    AUTHOR = {Lange, Herbert and Birkenhake, Christina},
     TITLE = {Complex abelian varieties},
    SERIES = {Grundlehren Math. Wiss.},
    VOLUME = {302},
 PUBLISHER = {Springer-Verlag, Berlin},
      YEAR = {1992},
     PAGES = {viii+435},
      ISBN = {3-540-54747-9},
   MRCLASS = {14-02 (14Kxx 32G20)},
  MRNUMBER = {1217487},
MRREVIEWER = {Olivier Debarre},
       DOI = {10.1007/978-3-662-02788-2},
       URL = {https://doi.org/10.1007/978-3-662-02788-2},
}

@article {bisatt,
    AUTHOR = {Bisatt, Matthew},
     TITLE = {Explicit root numbers of abelian varieties},
   JOURNAL = {Trans. Amer. Math. Soc.},
  FJOURNAL = {Transactions of the American Mathematical Society},
    VOLUME = {372},
      YEAR = {2019},
    NUMBER = {11},
     PAGES = {7889--7920},
      ISSN = {0002-9947},
   MRCLASS = {11G40 (11G10)},
  MRNUMBER = {4029685},
MRREVIEWER = {Timo Keller},
       DOI = {10.1090/tran/7926},
       URL = {https://doi-org.scd-rproxy.u-strasbg.fr/10.1090/tran/7926},
}

@book {bourbaki_alg1,
    AUTHOR = {Bourbaki, N.},
     TITLE = {\'{E}l\'{e}ments de math\'{e}matique. {A}lg\`ebre. {C}hapitres 1 \`a 3},
 PUBLISHER = {Hermann, Paris},
      YEAR = {1970},
     PAGES = {pp.~xiii+635 (not consecutively paged)},
   MRCLASS = {00.00 (13.00)},
  MRNUMBER = {0274237},
MRREVIEWER = {P. Samuel},
}

@book {bh_llc,
    AUTHOR = {Bushnell, Colin J. and Henniart, Guy},
     TITLE = {The local {L}anglands conjecture for {$\GL(2)$}},
    SERIES = {Grund\-leh\-ren Math. Wiss.},
    VOLUME = {335},
 PUBLISHER = {Springer-Verlag, Berlin},
      YEAR = {2006},
     PAGES = {xii+347},
   MRCLASS = {22E50 (11-02 11S37 22-02)},
  MRNUMBER = {2234120},
MRREVIEWER = {Alexandru Ioan Badulescu},
       DOI = {10.1007/3-540-31511-X},
}

@article {chai_semiab,
    AUTHOR = {Chai, Ching-Li},
     TITLE = {N\'eron models for semiabelian varieties: congruence and change
              of base field},
      NOTE = {Loo-Keng Hua: a great mathematician of the twentieth century},
   JOURNAL = {Asian J. Math.},
  FJOURNAL = {Asian Journal of Mathematics},
    VOLUME = {4},
      YEAR = {2000},
    NUMBER = {4},
     PAGES = {715--736},
%      ISSN = {1093-6106},
       DOI = {10.4310/AJM.2000.v4.n4.a1},
}

@article {chai_rm_end,
    AUTHOR = {Chai, Ching-Li},
     TITLE = {Every ordinary symplectic isogeny class in positive
              characteristic is dense in the moduli},
   JOURNAL = {Invent. Math.},
  FJOURNAL = {Inventiones Mathematicae},
    VOLUME = {121},
      YEAR = {1995},
    NUMBER = {3},
     PAGES = {439--479},
      ISSN = {0020-9910},
   MRCLASS = {11G10 (11G25 14G20 14K02 14K10 14K15)},
  MRNUMBER = {1353306},
MRREVIEWER = {F. Oort},
       DOI = {10.1007/BF01884309},
       URL = {https://doi-org.scd-rproxy.u-strasbg.fr/10.1007/BF01884309},
}

@article {ellenberg_end,
    AUTHOR = {Ellenberg, Jordan S.},
     TITLE = {Endomorphism algebras of {J}acobians},
   JOURNAL = {Adv. Math.},
  FJOURNAL = {Advances in Mathematics},
    VOLUME = {162},
      YEAR = {2001},
    NUMBER = {2},
     PAGES = {243--271},
      ISSN = {0001-8708},
   MRCLASS = {11G10},
  MRNUMBER = {1859248},
       DOI = {10.1006/aima.2001.1994},
       URL = {https://doi.org/10.1006/aima.2001.1994},
}

@article {connell,
    AUTHOR = {Connell, Ian},
     TITLE = {Calculating root numbers of elliptic curves over {${\mathbb Q}$}},
   JOURNAL = {Ma\-nus\-crip\-ta Math.},
  FJOURNAL = {Ma\-nus\-crip\-ta Mathematica},
    VOLUME = {82},
      YEAR = {1994},
    NUMBER = {1},
     PAGES = {93--104},
      ISSN = {0025-2611},
   MRCLASS = {11G05},
  MRNUMBER = {1254143},
MRREVIEWER = {Franck Lepr\'{e}vost},
       DOI = {10.1007/BF02567689},
       URL = {https://doi-org.scd-rproxy.u-strasbg.fr/10.1007/BF02567689},
}

@incollection {deligne_eq_fonctionelle,
    AUTHOR = {Deligne, Pierre},
     TITLE = {Les constantes des \'equations fonctionnelles des fonctions
              {$L$}},
 BOOKTITLE = {Modular functions of one variable, {II} ({P}roc. {I}nternat.
              {S}ummer {S}chool, {U}niv. {A}ntwerp, {A}ntwerp, 1972)},
     PAGES = {501--597},
    SERIES = {Lecture Notes in Math.}, 
    VOLUME = {349},
 PUBLISHER = {Springer, Berlin},
      YEAR = {1973},
   MRCLASS = {12A70 (10H10 12A65 12B25 12B30)},
  MRNUMBER = {0349635},
MRREVIEWER = {Roger E. Howe},
}

@article {dd_root_ellc2,
    AUTHOR = {Dokchitser, Tim and Dokchitser, Vladimir},
     TITLE = {Root numbers of elliptic curves in residue characteristic 2},
   JOURNAL = {Bull. Lond. Math. Soc.},
  FJOURNAL = {Bulletin of the London Mathematical Society},
    VOLUME = {40},
      YEAR = {2008},
    NUMBER = {3},
     PAGES = {516--524},
      ISSN = {0024-6093},
   MRCLASS = {11G05 (11F80 11G07 11G40)},
  MRNUMBER = {2418807},
MRREVIEWER = {Anupam Saikia},
       DOI = {10.1112/blms/bdn034},
       URL = {https://doi-org.scd-rproxy.u-strasbg.fr/10.1112/blms/bdn034},
}

@article {frohlich_queyrut,
    AUTHOR = {Fr\"{o}hlich, A. and Queyrut, J.},
     TITLE = {On the functional equation of the {A}rtin {$L$}-function for
              characters of real representations},
   JOURNAL = {Invent. Math.},
  FJOURNAL = {Inventiones Mathematicae},
    VOLUME = {20},
      YEAR = {1973},
     PAGES = {125--138},
      ISSN = {0020-9910},
   MRCLASS = {10H10},
  MRNUMBER = {0321888},
MRREVIEWER = {Stephen V. Ullom},
       DOI = {10.1007/BF01404061},
       URL = {https://doi-org.scd-rproxy.u-strasbg.fr/10.1007/BF01404061},
}

@article {iwasawa,
    AUTHOR = {Iwasawa, Kenkichi},
     TITLE = {On {G}alois groups of local fields},
   JOURNAL = {Trans. Amer. Math. Soc.},
  FJOURNAL = {Transactions of the American Mathematical Society},
    VOLUME = {80},
      YEAR = {1955},
     PAGES = {448--469},
      ISSN = {0002-9947},
   MRCLASS = {10.2X},
  MRNUMBER = {0075239},
MRREVIEWER = {C. Chevalley},
       DOI = {10.2307/1992998},
       URL = {https://doi-org.scd-rproxy.u-strasbg.fr/10.2307/1992998},
}

@incollection {knapp,
    AUTHOR = {Knapp, A. W.},
     TITLE = {Introduction to the {L}anglands program},
 BOOKTITLE = {Representation theory and automorphic forms ({E}dinburgh,
              1996)},
    SERIES = {Proc. Sympos. Pure Math.},
    VOLUME = {61},
     PAGES = {245--302},
 PUBLISHER = {Amer. Math. Soc., Providence, RI},
 year = {1997},
   MRCLASS = {11R39 (11F70 11S37 22E55)},
  MRNUMBER = {1476501},
MRREVIEWER = {Joe Repka},
}

@article {kraus,
    AUTHOR = {Kraus, Alain},
     TITLE = {Sur le d\'{e}faut de semi-stabilit\'{e} des courbes elliptiques \`a
              r\'{e}duction additive},
   JOURNAL = {Manuscripta Math.},
  FJOURNAL = {Manuscripta Mathematica},
    VOLUME = {69},
      YEAR = {1990},
    NUMBER = {4},
     PAGES = {353--385},
      ISSN = {0025-2611},
   MRCLASS = {11G07},
  MRNUMBER = {1080288},
MRREVIEWER = {Philippe Satg\'{e}},
       DOI = {10.1007/BF02567933},
       URL = {https://doi-org.scd-rproxy.u-strasbg.fr/10.1007/BF02567933},
}

@article {kobayashi,
    AUTHOR = {Kobayashi, Shinichi},
     TITLE = {The local root number of elliptic curves with wild
              ramification},
   JOURNAL = {Math. Ann.},
  FJOURNAL = {Mathematische Annalen},
    VOLUME = {323},
      YEAR = {2002},
    NUMBER = {3},
     PAGES = {609--623},
      ISSN = {0025-5831},
   MRCLASS = {11G07 (11G40 11R32)},
  MRNUMBER = {1923699},
MRREVIEWER = {Anupam Saikia},
       DOI = {10.1007/s002080200318},
       URL = {https://doi-org.scd-rproxy.u-strasbg.fr/10.1007/s002080200318},
}

@misc{lmfdb,
  shorthand    = {LMFDB},
  author       = {{The LMFDB Collaboration}},
  title        = {The {L}-functions and Modular Forms Database},
  howpublished = {\url{http://www.lmfdb.org}},
  year         = {2021},
  note         = {[Online; accessed 29 January 2021]},
}

@incollection {mestre_real_m,
    AUTHOR = {Mestre, J.-F.},
     TITLE = {Familles de courbes hyperelliptiques \`a multiplications
              r\'{e}elles},
 BOOKTITLE = {Arithmetic algebraic geometry ({T}exel, 1989)},
    SERIES = {Progr. Math.},
    VOLUME = {89},
     PAGES = {193--208},
 PUBLISHER = {Birkh\"{a}user Boston, Boston, MA},
      YEAR = {1991},
   MRCLASS = {14H35 (11R32 14H40 14H52 14K02)},
  MRNUMBER = {1085260},
MRREVIEWER = {J. Top},
}

@book {milneAG_book,
    AUTHOR = {Milne, J. S.},
     TITLE = {Algebraic groups},
    SERIES = {Cambridge Stud. Adv. Math.},
    VOLUME = {170},
      NOTE = {The theory of group schemes of finite type over a field},
 PUBLISHER = {Cambridge Univ. Press, Cambridge},
      YEAR = {2017},
     PAGES = {xvi+644},
%      ISBN = {978-1-107-16748-3},
   MRCLASS = {14L15 (17B45 20-01 20Dxx 20G40)},
  MRNUMBER = {3729270},
       DOI = {10.1017/9781316711736},
%       URL = {https://doi-org.scd-rproxy.u-strasbg.fr/10.1017/9781316711736},
}

@book {mumford,
    AUTHOR = {Mumford, David},
     TITLE = {Abelian varieties},
    SERIES = {Tata Inst. Fund. Res. Stud. on Math. and Phys.,
              No. 5 },
 PUBLISHER = {Published for the Tata Institute of Fundamental Research,
              Bombay; Oxford University Press, London},
      YEAR = {1970},
     PAGES = {viii+242},
   MRCLASS = {14.51},
  MRNUMBER = {0282985},
MRREVIEWER = {J. S. Milne},
}

@article {nekovar_comp_l,
    AUTHOR = {Nekov\'{a}\v{r}, Jan},
     TITLE = {Compatibility of arithmetic and algebraic local constants (the
              case {$\ell\ne p$})},
   JOURNAL = {Compos. Math.},
  FJOURNAL = {Compositio Mathematica},
    VOLUME = {151},
      YEAR = {2015},
    NUMBER = {9},
     PAGES = {1626--1646},
      ISSN = {0010-437X},
   MRCLASS = {11S40 (11G05 11G40 14G25)},
  MRNUMBER = {3406439},
MRREVIEWER = {Jeanine Van Order},
       DOI = {10.1112/S0010437X14008069},
       URL = {https://doi-org.scd-rproxy.u-strasbg.fr/10.1112/S0010437X14008069},
}

@article {nekovar_comp_bt,
    AUTHOR = {Nekov\'{a}\v{r}, Jan},
     TITLE = {Compatibility of arithmetic and algebraic local constants,
              {II}: the tame abelian potentially {B}arsotti-{T}ate case},
   JOURNAL = {Proc. Lond. Math. Soc. (3)},
  FJOURNAL = {Proceedings of the London Mathematical Society. Third Series},
    VOLUME = {116},
      YEAR = {2018},
    NUMBER = {2},
     PAGES = {378--427},
      ISSN = {0024-6115},
   MRCLASS = {11S40 (11G05 11G40 14G25)},
  MRNUMBER = {3764064},
MRREVIEWER = {Julio Andrade},
       DOI = {10.1112/plms.12085},
       URL = {https://doi-org.scd-rproxy.u-strasbg.fr/10.1112/plms.12085},
}

@book {rz,
    AUTHOR = {Ribes, Luis and Zalesskii, Pavel},
     TITLE = {Profinite groups},
    SERIES = {Ergeb. Math. Grenzgeb. (3)},
    VOLUME = {40},
   EDITION = {Second},
 PUBLISHER = {Springer-Verlag, Berlin},
      YEAR = {2010},
     PAGES = {xvi+464},
      ISBN = {978-3-642-01641-7},
   MRCLASS = {20E18},
  MRNUMBER = {2599132},
       DOI = {10.1007/978-3-642-01642-4},
       URL = {https://doi-org.scd-rproxy.u-strasbg.fr/10.1007/978-3-642-01642-4},
}

@article {ribet_real,
    AUTHOR = {Ribet, Kenneth A.},
     TITLE = {Galois action on division points of {A}belian varieties with
              real multiplications},
   JOURNAL = {Amer. J. Math.},
  FJOURNAL = {American Journal of Mathematics},
    VOLUME = {98},
      YEAR = {1976},
    NUMBER = {3},
     PAGES = {751--804},
   MRCLASS = {14K15},
  MRNUMBER = {0457455},
MRREVIEWER = {Jacques Velu},
       DOI = {10.2307/2373815},
}

@incollection {rohrlich,
    AUTHOR = {Rohrlich, David E.},
     TITLE = {{E}lliptic curves and the {W}eil-{D}eligne group},
 BOOKTITLE = {Elliptic curves and related topics},
    SERIES = {CRM Proc. Lecture Notes},
    VOLUME = {4},
     PAGES = {125--157},
 PUBLISHER = {Amer. Math. Soc., Providence, RI},
      YEAR = {1994},
   MRCLASS = {11G07 (11F80 11F85)},
  MRNUMBER = {1260960},
MRREVIEWER = {Henri Darmon},
}

@article {rohrlich_formulas,
    AUTHOR = {Rohrlich, David E.},
     TITLE = {Galois theory, elliptic curves, and root numbers},
   JOURNAL = {Compositio Math.},
  FJOURNAL = {Compositio Mathematica},
    VOLUME = {100},
      YEAR = {1996},
    NUMBER = {3},
     PAGES = {311--349},
   MRCLASS = {11G05 (11F80 11G07 11G40 11R32)},
  MRNUMBER = {1387669},
MRREVIEWER = {Kenneth Kramer},
}

@article {rohrlich_galois_inv,
    AUTHOR = {Rohrlich, David E.},
     TITLE = {Galois invariance of local root numbers},
   JOURNAL = {Math. Ann.},
  FJOURNAL = {Mathematische Annalen},
    VOLUME = {351},
      YEAR = {2011},
    NUMBER = {4},
     PAGES = {979--1003},
      ISSN = {0025-5831},
   MRCLASS = {11R42 (11F80 11R32 11S20)},
  MRNUMBER = {2854120},
MRREVIEWER = {Th\cfac{o}ng Nguy\cftil{e}n-Quang-\Dbar \cftil{o}},
       DOI = {10.1007/s00208-010-0626-z},
       URL = {https://doi-org.scd-rproxy.u-strasbg.fr/10.1007/s00208-010-0626-z},
}

@article {rohrlich_proof,
    AUTHOR = {Rohrlich, David E.},
     TITLE = {Variation of the root number in families of elliptic curves},
   JOURNAL = {Compositio Math.},
  FJOURNAL = {Compositio Mathematica},
    VOLUME = {87},
      YEAR = {1993},
    NUMBER = {2},
     PAGES = {119--151},
     CODEN = {CMPMAF},
   MRCLASS = {11G40 (11G05 11N36)},
  MRNUMBER = {1219633},
MRREVIEWER = {Fernando Q. Gouv{\^e}a},
}

@article {sabitova_root,
    AUTHOR = {Sabitova, Maria},
     TITLE = {Root numbers of abelian varieties},
   JOURNAL = {Trans. Amer. Math. Soc.},
  FJOURNAL = {Transactions of the American Mathematical Society},
    VOLUME = {359},
      YEAR = {2007},
    NUMBER = {9},
     PAGES = {4259--4284},
%      ISSN = {0002-9947},
   MRCLASS = {11G10 (11F80 11G40 11R32)},
  MRNUMBER = {2309184},
MRREVIEWER = {Jae-Hyun Yang},
       DOI = {10.1090/S0002-9947-07-04148-7},
%       URL = {https://doi-org.scd-rproxy.u-strasbg.fr/10.1090/S0002-9947-07-04148-7},
}

@inproceedings {serre_Lfun_conj,
    AUTHOR = {Serre, Jean-Pierre},
     TITLE = {Facteurs locaux des fonctions z\^{e}ta des variet\'{e}s alg\'{e}briques
              (d\'{e}finitions et conjectures)},
 BOOKTITLE = {S\'{e}minaire {D}elange-{P}isot-{P}oitou. 11e ann\'{e}e: 1969/70.
              {T}h\'{e}orie des nombres. {F}asc. 1: {E}xpos\'{e}s 1 \`a 15; {F}asc. 2:
              {E}xpos\'{e}s 16 \`a 24},
     PAGES = {15},
 PUBLISHER = {Secr\'{e}tariat Math., Paris},
      YEAR = {1970},
   MRCLASS = {11F67 (11G10 14G10)},
  MRNUMBER = {3618526},
}

@article {serre_tate,
    AUTHOR = {Serre, Jean-Pierre and Tate, John},
     TITLE = {Good reduction of abelian varieties},
   JOURNAL = {Ann. of Math. (2)},
  FJOURNAL = {Annals of Mathematics. Second Series},
    VOLUME = {88},
      YEAR = {1968},
     PAGES = {492--517},
   MRCLASS = {14.51},
  MRNUMBER = {0236190},
MRREVIEWER = {M. J. Greenberg},
}

@book {sga7_i,
     TITLE = {Groupes de monodromie en g\'{e}om\'{e}trie alg\'{e}brique. {I}},
    SERIES = {S\'{e}minaire de G\'{e}om\'{e}trie Alg\'{e}brique du Bois-Marie 1967--1969
              (SGA 7 I),
              Dirig\'{e} par A. Grothendieck. Avec la collaboration de M.
              Raynaud et D. S. Rim. Lecture Notes in Mathematics, Vol. 288},
      NOTE = {S\'{e}minaire de G\'{e}om\'{e}trie Alg\'{e}brique du Bois-Marie 1967--1969
              (SGA 7 I),
              Dirig\'{e} par A. Grothendieck. Avec la collaboration de M.
              Raynaud et D. S. Rim},
 PUBLISHER = {Springer-Verlag, Berlin-New York},
 shorthand = {SGA 7.I},
 sorname = {SGA 7.1},
      YEAR = {1972},
     PAGES = {viii+523},
   MRCLASS = {14-06},
  MRNUMBER = {0354656},
}

@book {shimura_taniyama,
    AUTHOR = {Shimura, Goro and Taniyama, Yutaka},
     TITLE = {Complex multiplication of abelian varieties and its
              applications to number theory},
    SERIES = {Publ. Math. Soc. Japan},
    VOLUME = {6},
 PUBLISHER = {The Mathematical Society of Japan, Tokyo},
      YEAR = {1961},
     PAGES = {xi+159},
   MRCLASS = {14.40 (10.68)},
  MRNUMBER = {0125113},
MRREVIEWER = {I. Barsotti},
}

@incollection {tate_background,
    AUTHOR = {Tate, J. T.},
     TITLE = {Number theoretic background},
 BOOKTITLE = {Automorphic forms, representations and {$L$}-functions
              ({P}roc. {S}ympos. {P}ure {M}ath., {O}regon {S}tate {U}niv.,
              {C}orvallis, {O}re., 1977), {P}art 2},
    SERIES = {Proc. Sympos. Pure Math., XXXIII},
     PAGES = {3--26},
 PUBLISHER = {Amer. Math. Soc., Providence, R.I.},
      YEAR = {1979},
   MRCLASS = {12A67},
  MRNUMBER = {546607},
MRREVIEWER = {A. I. Vinogradov},
}

@article {tautz_real_m,
    AUTHOR = {Tautz, Walter and Top, Jaap and Verberkmoes, Alain},
     TITLE = {Explicit hyperelliptic curves with real multiplication and
              permutation polynomials},
   JOURNAL = {Canad. J. Math.},
  FJOURNAL = {Canadian Journal of Mathematics. Journal Canadien de
              Math\'{e}matiques},
    VOLUME = {43},
      YEAR = {1991},
    NUMBER = {5},
     PAGES = {1055--1064},
      ISSN = {0008-414X},
   MRCLASS = {11G30 (14H30 14H35)},
  MRNUMBER = {1138583},
MRREVIEWER = {Philippe Satg\'{e}},
       DOI = {10.4153/CJM-1991-061-x},
       URL = {https://doi.org/10.4153/CJM-1991-061-x},
}

\end{document}